\documentclass[11pt, reqno]{amsart}
\usepackage[latin1]{inputenc}
\usepackage{amssymb}
\usepackage{pdfsync}
\usepackage[english]{babel}

\usepackage{yhmath}

\usepackage{graphicx}
\usepackage[normalem]{ulem}
\usepackage{fullpage}
\usepackage{color}
\usepackage{amsmath}
\usepackage{amsfonts}
\usepackage{mathrsfs}
\usepackage{t1enc , graphicx}
\usepackage{verbatim}
\usepackage{bbm}
\usepackage{enumitem}
\usepackage{tikz-cd}
\usepackage{capt-of}
\usepackage[pdftex, colorlinks=true,urlcolor=blue,linkcolor=blue,citecolor=blue, linktocpage]{hyperref}
\usepackage[capitalize]{cleveref}

\usepackage[left= 1.2 in, right= 1.2 in ,top= 1.2 in, bottom = 2.1 in]{geometry}

\newcommand{\R}{\mathbb{R}}

\newcommand{\C}{\mathbb{C}}

\newcommand{\Z}{\mathbb{Z}}
\newcommand{\N}{\mathbb{N}}

\newcommand{\E}{\mathbb{E}}

\newcommand{\bX}{\mathbf{X}}
\newcommand{\bY}{\mathbf{Y}}
\newcommand{\bZ}{\mathbf{Z}}
\newcommand{\bK}{\mathbf{K}}

\newcommand{\calB}{\mathcal{B}}
\newcommand{\calD}{\mathcal{D}}
\newcommand{\calE}{\mathcal{E}}

\newcommand{\wL}{\widehat{\Lambda}}

\newcommand{\mb}{\mathbf}

\newcommand{\supp}{\operatorname{supp}}

\newcommand\redsout{\bgroup\markoverwith{\textcolor{red}{\rule[0.5ex]{2pt}{0.4pt}}}\ULon}

\newtheorem{theorem}{Theorem}[section]
\newtheorem*{theorem*}{Theorem}
\newtheorem{proposition}[theorem]{Proposition}
\newtheorem{lemma}[theorem]{Lemma}
\newtheorem{claim}[theorem]{Claim}
\newtheorem{corollary}[theorem]{Corollary}
\newtheorem*{corollary*}{Corollary}
\newtheorem{question}[theorem]{Question}
\newtheorem{conjecture}[theorem]{Conjecture}

\theoremstyle{definition}
\newtheorem*{definition*}{Definition}
\newtheorem{definition}[theorem]{Definition}

\theoremstyle{remark}

\newtheorem*{remark*}{Remark}
\newtheorem{remark}[theorem]{Remark}

\theoremstyle{plain}
\newcounter{MainTheoremCounter}

\theoremstyle{plain}
\newcounter{OldTheoremCounter}

\newtheorem{Oldtheorem}[OldTheoremCounter]{Theorem}

\newcommand{\vertiii}[1]{{\left\vert\kern-0.25ex\left\vert\kern-0.25ex\left\vert #1 
    \right\vert\kern-0.25ex\right\vert\kern-0.25ex\right\vert}}

\newcommand{\AuthornoteA}[2]{{\sf\small\color{magenta}{[#1: #2]}}}
\newcommand{\AuthornoteH}[2]{{\sf\small\color{blue}{[#1: #2]}}}

\newcommand{\Hnote}{\AuthornoteH{H}}
\newcommand{\Anote}{\AuthornoteA{A}}

\begin{document}
\author{John T. Griesmer}    
 \author{Anh N. Le}
 \author{Th\'ai Ho\`ang L\^e}

\address{Department of Applied Mathematics and Statistics\\
Colorado School of Mines\\
1005 14th Street, Golden, CO 80401} 
\email{jtgriesmer@gmail.com}

\address{Department of Mathematics\\
	Ohio State University\\
	231 W. 18th Ave., Columbus, OH 43210}
\email{le.286@osu.edu}

\address{Department of Mathematics\\
    University of Mississippi\\
    University, MS 38677, USA}
\email{leth@olemiss.edu}

\title{Bohr sets in sumsets II: countable abelian groups}


\maketitle

\begin{abstract}
    We prove three results concerning the existence of Bohr sets in threefold sumsets. More precisely, letting $G$ be a countable discrete abelian group and $\phi_1, \phi_2, \phi_3: G \to G$ be commuting endomorphisms whose images have finite indices, we show that
    \begin{enumerate}
        \item If $A \subset G$ has positive upper Banach density and $\phi_1 + \phi_2 + \phi_3 = 0$, then $\phi_1(A) + \phi_2(A) + \phi_3(A)$ contains a Bohr set. This generalizes a theorem of Bergelson and Ruzsa in $\mathbb{Z}$ and a recent result of the first author.
        
        \item For any partition $G = \bigcup_{i=1}^r A_i$, there exists an $i \in \{1, \ldots, r\}$ such that $\phi_1(A_i) + \phi_2(A_i) - \phi_2(A_i)$ contains a Bohr set. This generalizes a result of the second and third authors from $\mathbb{Z}$ to countable abelian groups.
        
        \item If $B, C \subset G$ have positive upper Banach density and $G = \bigcup_{i=1}^r A_i$ is a partition, $B + C + A_i$ contains a Bohr set for some $i \in \{1, \ldots, r\}$. This is a strengthening of a theorem of Bergelson, Furstenberg, and Weiss.
    \end{enumerate}
    All results are quantitative in the sense that the radius and rank of the Bohr set obtained depends only on the indices $[G:\phi_j(G)]$, the upper Banach density of $A$ (in (1)), or the number of sets in the given partition (in (2) and (3)). 
\end{abstract}

\tableofcontents

\section{Introduction}
This paper continues the investigation set forth in \cite{ll}. Let $G$ be an abelian topological group. If $A, B \subset G$, the \emph{sumset} and \emph{difference set} of $A$ and $B$ are $A + B: = \{a + b: a \in A, b \in B \}$ and $A - B:=\{a - b: a \in A, b \in B \}$, respectively.  For $a\in G$, the \emph{translate} $a+B$ is $\{a+B:b\in B\}$. If $s \in \Z$, we define $sA:=\{sa: a \in A\}$.  A \emph{character} of $G$ is a continuous homomorphism from $G$ to $S^1:=\{z\in \mathbb C: |z|=1\}$.

Many classical results in additive combinatorics state, roughly, that sumsets are more structured than their summands.  Such results often quantify the structure found in sumsets in terms of \emph{Bohr sets}, which we define here.
For a finite set $\Lambda$ of characters of $G$  and a constant $\eta > 0$, the set 
\begin{equation*} \label{eq:bohr1}
B(\Lambda; \eta) := \{ x \in G : | \gamma(x)-1 | < \eta \textup{ for all } \gamma \in \Lambda\}
\end{equation*} 
is called a \emph{Bohr set}, a \emph{Bohr$_0$-set}, or a \emph{Bohr neighborhood of $0$} in the literature. In this paper we use mostly the first nomenclature. The set $B(\Lambda; \eta)$ is also called a Bohr-$(k, \eta)$ set where $k = |\Lambda|$. We refer to $\eta$ as the \emph{radius} and $k$ as the \emph{rank} of the Bohr set. By a \textit{translate of a Bohr set}, or a \textit{Bohr neighborhood}, we mean a set of the form $a + B(\Lambda; \eta)$ for some $a \in G$.



After summarizing previous results in Sections \ref{sec:PreviousZ} and \ref{sec:PreviousCompact}, we state our new results in Section \ref{sec:NewResults}.

\subsection{Previous results in \texorpdfstring{$\Z$}{Z}}\label{sec:PreviousZ} 

If $A \subset \Z$, the \textit{upper Banach density} of $A$ is
\begin{equation*} \label{eq:density_Z}
    d^*(A) =  \limsup_{N \to \infty} \max_{M \in \Z} \frac{|A \cap \{M+1, \ldots, M + N\}|}{N}.
\end{equation*}
The study of Bohr sets in sumsets started with the following important theorem of Bogolyubov \cite{bogo}. 

\begin{Oldtheorem}[Bogolyubov] \label{th:bogo}
If $A \subset \Z$ has positive upper Banach density, 
then $A - A + A - A$ contains a Bohr set whose rank and radius depend only on $d^*(A)$.
\end{Oldtheorem}

While it originated from the study of almost periodic functions, Bogolyubov's theorem is now a standard tool in additive combinatorics. It was used in Ruzsa's proof of Freiman's theorem \cite{ruzsa6} and in Gowers's proof of Szemer\'edi's theorem \cite{gowers}.  

 F{\o}lner \cite{folner1} showed that the last two summands in Bogolyubov's theorem are ``almost'' redundant by proving that $A - A$ already contains a set of the form $B\setminus E$, where $B$ is a Bohr set and $d^*(E)=0$. 
The exceptional set $E$ is unavoidable: Kriz  \cite{kriz} demonstrated that there exists a set $A$ of positive upper Banach density for which $A - A$ contains no Bohr sets. The first author \cite{griesmer-separating-bohr} showed that there is a set $A$ having $d^*(A)>0$ such that $A-A$ contains no Bohr neighborhood of any integer.  

Hegyv\'ari  and Ruzsa \cite{hr} generalized Bogolyubov's theorem in a different direction, showing that there exist ``many'' $a \in \Z$ for which $A - A + A - a$ contains a Bohr set. Bj\"orklund and the first author \cite[Theorem 1.1]{bg} strengthened this result by providing explicit bounds on the rank and radius of such a Bohr set, and generalized the result to all countable amenable discrete groups (and hence all countable discrete abelian groups).

Regarding more general threefold sumsets, Bergelson and Ruzsa proved the following:

\begin{Oldtheorem}[{\cite[Theorem 6.1]{br}}] 
\label{th:br}
Let $s_1, s_2, s_3$ be non-zero integers satisfying $s_1 +s_2+s_3 = 0$. If $A \subset \Z$ has positive upper Banach density, then $s_1 A+s_2 A+s_3 A$ contains a Bohr set whose rank and radius depend only on $s_1, s_2, s_3$ and $d^*(A)$. 
\end{Oldtheorem} 
Since any Bohr set in $\Z$ must contain $0$, the condition $s_1 + s_2 + s_3 = 0$ is easily seen to be necessary by taking $A = M \Z + 1$ for some $M > |s_1| + |s_2| + |s_3|$. In particular, one cannot expect $A + A - A$ to contain a Bohr set for every $A$ of positive upper Banach density. When $(s_1, s_2, s_3) = (1, 1, -2)$, \cref{th:br}  generalizes \cref{th:bogo}, since $A+A-2 A \subset A+A-A-A$.   

While the problem of finding Bohr sets in sumsets where the summands have positive upper Banach density has attracted much attention, the analogous question concerning partitions was little studied until recently, and the situation is less well understood. The following question, popularized by Katznelson \cite{katznelson} and Ruzsa \cite[Chapter 5]{ruzsabook}, is a well-known open problem in additive combinatorics and dynamical systems.

\begin{question} \label{q:kr}
If $\Z = \bigcup_{i=1}^r A_i$, must one of the difference sets $A_i - A_i$ contain a Bohr set?
\end{question}
In terms of dynamical systems, Question \ref{q:kr} asks if every set of recurrence for minimal \textit{isometries} (also known as a set of \textit{Bohr recurrence}) is also a set of recurrence for minimal \textit{topological systems}. See \cite{gkr} for a detailed account of the history of \cref{q:kr} and many equivalent formulations.  See \cite{griesmer-special-cases} for more equivalent formulations and resolution of some special cases.

Regarding three summands, the second and third authors proved the following partition analogue of \cref{th:br}.
\begin{Oldtheorem}[{\cite[Theorem 1.4]{ll}}] \label{th:ll}\ 
 \begin{enumerate}[label=(\roman*)]
 \item   Let $s_1, s_2 \in \Z \setminus \{0\}$. For any partition $\Z = \bigcup_{i=1}^r A_i$, there is an $i$ such that $s_1 A_i + s_2 A_i - s_2 A_i$ contains a Bohr set whose rank and radius depend only on $ s_1, s_2$ and $r$.
\item For any partition $\Z = \bigcup_{i=1}^r A_i$, there is an $i$ such that $A_i - A_i + sA_i$ contains a Bohr set for any $s \in \Z \setminus \{0\}$.
    \end{enumerate}
\end{Oldtheorem}
Rado's theorem says that an equation $\sum_{j=1}^k s_j x_j = 0$ with coefficients $s_j \in \Z\setminus \{0\}$ is partition regular over $\Z \setminus \{0\}$ if and only if there exists $J \subset \{1, \ldots, k\}, J \neq \varnothing$ such that $\sum_{j \in J}s_j = 0$.
Combined with \cref{th:br}, part (i) of Theorem \ref{th:ll} gives a complete characterization of tuples $(s_1, \ldots, s_k) \in (\Z \setminus \{0\})^k$ that guarantee the existence of a Bohr set in $\sum_{j=1}^k s_j A_i$, for some $i$, as long as $k \geq 3$: They are precisely tuples satisfying Rado's condition.\footnote{To see that this condition is necessary, suppose $\sum_{j=1}^k s_j A_i$ contains a Bohr set. By giving $0$ its own partition class, we may assume $0 \not \in A_i$. Since a Bohr set must necessarily contain $0$, this implies that there are $x_j \in A_i$ such that $\sum_{j=1}^k s_j x_j =0$, and Rado's condition applies. To see that this condition is sufficient, observe that $(s+t)A \subset sA + tA$, so the case $k \geq 3$ can be reduced to the case $k=3$.}
 This characterization is a strengthening of Rado's theorem. As the integer $s$ in Part (ii) can be arbitrarily large, this suggests that either the answer to \cref{q:kr} is positive, or the construction of a counterexample must be very delicate. 

\subsection{Previous results in compact groups}\label{sec:PreviousCompact}
As part of a general program, we aim to study the Bohr sets in sumsets phenomenon in more general groups. A natural setup is \textit{amenable groups}, since in these groups there is a natural notion of density, and Bohr sets can also be defined.\footnote{For non-abelian groups $G$, Bohr sets can be defined in terms of finite-dimensional unitary irreducible representations of $G$, see \cite{bg}.} A locally compact group $G$ with left Haar measure $m_G$ is said to be amenable if there exists an \textit{invariant mean} on $G$, that is, a linear functional $\lambda$ on $L^\infty(m_G)$ that is nonnegative (i.e.~$\lambda(f) \geq 0$ if $f \geq 0$), of norm 1 (i.e.~$\lambda(1_G)=1$) and left-invariant (i.e.~$\lambda(f_t) = \lambda(f)$, where $f_t(x) = f(t^{-1}x))$. If $A \subset G$ is a Borel set, we can define its \textit{upper Banach density} as
\begin{equation} \label{eq:density2}
    d^*(A) = \sup\{ \lambda(1_A): \textup{ $\lambda$ is an invariant mean on $G$.}\}
\end{equation}
The supremum is actually a maximum, since the set of invariant means on $G$ is weak*-compact, by the Banach-Alaoglu theorem. It is well known that all locally compact abelian groups are amenable. F\o lner \cite{folner1, folner2} generalized \cref{th:bogo} to discrete abelian groups, and the results of \cite{bg} mentioned above apply to countable discrete amenable groups which are not necessarily abelian.

Against this backdrop, our objective in this program is threefold. First, we ask for analogues of Theorems \ref{th:br} and \ref{th:ll} in (a subclass of) amenable groups. Second, in the context of general groups, we can replace the dilate $sA$ by $\phi(A)$, the image of $A$ under a homomorphism $\phi$. This point of view leads to a wider range of applications: we can consider linear maps on vector spaces and multiplication by an element in a ring (see Corollary \ref{cor:numberfields} below). This broader perspective was also adopted in recent works \cite{abb, abo} on Khintchine-type recurrence for actions of an abelian group. Third, we aim for \textit{uniformity} in terms of rank and radius of the Bohr set in question, i.e., they are allowed to depend on $d^*(A)$ and other parameters, but not $A$ itself. This is because, in some situations, the existence of Bohr sets is straightforward (for example, an interval around 0 in $\R/\Z$ always contains a Bohr set), but obtaining uniformity is much harder.

In \cite{ll}, these objectives were achieved for compact abelian groups. Note that in this case, the only invariant mean on $G$ is given by $m_G$ (the normalized Haar measure on $G$) and $d^*(A) = m_G(A)$. The second and third authors proved the following.

\begin{Oldtheorem}[Le-L\^e \cite{ll}] \label{th:ll_compact}
Let $K$ be a compact abelian group with normalized Haar measure $m_K$. Let $\phi_1, \phi_2, \phi_3: K \to K$ be commuting continuous endomorphisms such that $[K: \phi_j(K)] < \infty$ for each $j$. 
\begin{enumerate}[label=(\roman*)]
    \item If $\phi_1 + \phi_2 +\phi_3=0$ and $A \subset K$ is a Borel set with $m_K(A) > 0$, then $\phi_1(A) + \phi_2(A) + \phi_3(A)$ contains a Bohr-$(k,\eta)$ set, where $k$ and $\eta$ depend only on $m_K(A)$ and $[G: \phi_j(G)]$.
    \item If $K = \bigcup_{i=1}^r A_i$ is a partition of $K$ into Borel sets, then there exists $i$ such that $\phi_1(A_i) + \phi_2(A_i)-\phi_2(A_i)$ contains a Bohr-$(k,\eta)$ set, where $k$ and $\eta$ depend only on $r$ and $[G: \phi_j(G)]$.
\end{enumerate}
\end{Oldtheorem}

The finite index condition is necessary and also appears in \cite{abb}. On the other hand, we do not know if the assumption that the $\phi_j$ commute can be omitted.

\subsection{New results in discrete groups}\label{sec:NewResults} In this paper we extend many of the preceding results to the setting of countable discrete abelian groups. Our main results are discrete analogues of \cref{th:ll_compact}, and as such are direct generalizations of Theorems \ref{th:br} and \ref{th:ll}.

\begin{theorem} \label{th:main-density}
Let $G$ be a countable discrete abelian group. Let $\phi_1, \phi_2, \phi_3 : G \rightarrow G$ be commuting endomorphisms such that $\phi_1 + \phi_2 +\phi_3 =0$ and $[G:\phi_j(G)]$ are finite for $j \in \{1, 2, 3\}$.
Suppose $A \subset G$ has positive upper Banach density, i.e.~$d^*(A) >0$. Then the set
\[
\phi_1(A) + \phi_2(A) + \phi_3(A)
\]
contains a Bohr-$(k, \eta)$ set, where $k$ and $\eta$ depend only on $d^*(A)$ and the indices $[G: \phi_j(G)]$.
\end{theorem}


\begin{remark}\
\begin{itemize}
\item In the special case $\phi_j(x) = s_j x$ where $s_j \in \Z \setminus\{0\}$, \cref{th:main-density} was proven by the first author \cite{griesmer-br} without the conclusion on the uniformity of $k$ and $\eta$.

\item The conclusion of \cref{th:main-density} remains valid if the $\phi_j$ do not necessarily commute, but one of them is an automorphism. 
Indeed, assume that $\phi_1$ is an automorphism. We observe that
\[
\phi_1(A) + \phi_2(A) + \phi_3(A) = \phi_1 \left(A + \phi_1^{-1} \circ \phi_2 (A) + \phi_1^{-1} \circ \phi_3 (A) \right).
\]
Consider the endomorphisms $Id$, $\phi_1^{-1} \circ \phi_2$ and $\phi_1^{-1} \circ \phi_3$. They add up to $0$ since
\[
    Id + \phi_1^{-1} \circ \phi_2 + \phi_1^{-1} \circ \phi_3 = Id + \phi_1^{-1} \circ (\phi_2 + \phi_3) = Id + \phi_1^{-1} \circ (- \phi_1) = 0.
\]
They also commute\footnote{Whenever three endomorphisms sum to $0$ and two of them commute, all three must commute.  Since $Id$ commutes with every endomorphism, these three commute.}, and have finite index images. 
\cref{th:main-density} implies $A + \phi_1^{-1} \circ \phi_2 (A) + \phi_1^{-1} \circ \phi_3 (A)$ contains a Bohr set, and the image of a Bohr set under an automorphism is easily seen to be a Bohr set of the same rank and radius (see \cref{lem:from_kronecker_up}).

\item The hypothesis $\phi_1 + \phi_2 + \phi_3 = 0$ cannot be removed as demonstrated in the remark after \cref{th:br}.

\item Similarly, the condition that each index $[G:\phi_j(G)]$ is finite cannot be omitted. For example, take $G = \Z$, $\phi_1(x) = x$, $\phi_2(x) = -x$, and $\phi_3(x)=0$ for $x \in \Z$. Then $\phi_1(A) + \phi_2(A) + \phi_3(A)= A-A$, and the Kriz example \cite{kriz} shows that there exists a set $A$ of positive upper Banach density such that $A-A$ does not contain any Bohr set.  See \cite[Remark 1.6]{griesmer-br} for further discussion.
\end{itemize}
\end{remark}


\begin{theorem} \label{th:main-partition}
Let $G$ be a discrete abelian group and let $\phi_1, \phi_2: G \rightarrow G$ be commuting endomorphisms such that $[G:\phi_j(G)]$ is finite for $j \in  \{1, 2\}$.
Then for every finite partition $G = \bigcup_{i=1}^r A_i$, there exists $i \in \{1, \ldots, r\}$ such that
\[
\phi_1(A_i) + \phi_2(A_i) - \phi_2(A_i)
\]
contains a Bohr-$(k, \eta)$ set, where $k$ and $\eta$ depend only on $r$ and the indices $[G: \phi_j(G)]$.
\end{theorem}
\begin{remark}\
\begin{itemize}
\item In contrast to Theorem \ref{th:main-density} and Theorem \ref{thm:B+C+A_i_discrete} below, \cref{th:main-partition} does not assume $G$ is countable. The reason is that the former two theorems use Kronecker factors via Furstenberg's correspondence principle, and the theory of factors requires the group to be countable. There are two ways to think of a factor of a measure preserving $G$-system: as a spatial map or as a $G$-invariant sub $\sigma$-algebra. The latter can be obtained trivially from the former, but the converse is not trivial, and requires the group to be countable (in addition to the $\sigma$-algebras being separable). For instance, the method of proof of Theorem 5.15 in \cite{Furstenberg-book} requires $G$ to be countable.

\item Since Bohr sets contain $0$, \cref{th:main-partition} implies that the equation $\phi_1(x) + \phi_2(y) - \phi_2(z) = 0$ is partition regular in discrete abelian groups, that is, under any partition $G = \bigcup_{i=1}^r A_i$, there exists non-zero $x, y, z$ in the same class $A_i$ such that $\phi_1(x) + \phi_2(y) - \phi_2(z) = 0$. (To see that we can take $x, y, z$ to be nonzero, give 0 its own partition class.)

\item If $d^*(A) >0$, then $A + A - A$ is not guaranteed to contain a Bohr set as remarked after \cref{th:br}. In particular, the analogous version of \cref{th:main-partition} for sets of positive upper Banach density is false. 

\item The hypothesis that $\phi_2(G)$ has finite index in $G$ cannot be omitted.  For example, taking $\phi_2 = 0$ and $\phi_1(x) = x$ for $x \in G$, the sumset in Theorem \ref{th:main-partition} simplifies to $A_i$.

The question of whether the Theorem \ref{th:main-partition} remains true without the assumption that $[G:\phi_1(G)]$ is finite is essentially \cref{q:kr}: we may take $\phi_1(x)=0$ and $\phi_2(x)=x$ for all $x\in G$, and the sumset in Theorem \ref{th:main-partition} simplifies to $A_i-A_i$.  

\item Similar to \cref{th:main-density}, the hypothesis that the $\phi_j$ commute can be removed if one of them is an automorphism. 

\end{itemize}
\end{remark}
As a consequence of Theorems \ref{th:main-density} and \ref{th:main-partition}, we obtain immediately the following number field generalization of Theorems \ref{th:br} and \ref{th:ll}. In \cite{ll}, this result was proved (at least for $\Z[i]$) using a different argument, similar to Bogolyubov and Bergelson-Ruzsa's proofs of Theorems \ref{th:bogo} and \ref{th:br} in $\Z$.

\begin{corollary} \label{cor:numberfields}
Let $K$ be an algebraic number field of degree $d$ and $\mathcal{O}_K$ be its ring of integers (so the additive group of $\mathcal{O}_K$ is isomorphic to $\Z^d$). Let $s_1, s_2, s_3 \in \mathcal{O}_K \setminus \{0\}$ such that $s_1+s_2+s_3=0$. 
\begin{enumerate}[label=(\roman*)]
    \item If $A \subset \mathcal{O}_K$ has $d^*(A) > 0$, then $s_1A + s_2A + s_3A$ contains a Bohr set, whose rank and radius depend only on $d^*(A)$ and the norms of $s_1, s_2, s_3$.
    \item If $\mathcal{O}_K = \bigcup_{i=1}^r A_i$, then there exists $i$ such that $s_1A_i + s_2A_i-s_2A_i$ contains a Bohr set, whose rank and radius depend only on $r$ and the norms of $s_1$ and $s_2$.
\end{enumerate}
\end{corollary}
    

Bergelson, Furstenberg, and Weiss \cite[Corollary 1.3]{bfw-bohr} showed that if $B, C \subset \Z$ have positive upper Banach density and $A \subset \Z$ is syndetic, then $B + C + A$ contains a translate of a Bohr set. Here a set $A \subset \Z$ is \emph{syndetic} if a collection of finitely many translates of $A$ covers $\Z$. Our next theorem not only generalizes Bergelson-Furstenberg-Weiss's result to countable abelian groups but also strengthens it by only assuming that $A$ arises from an arbitrary partition. Moreover, we provide quantitative bounds on the radius and rank of the Bohr set, a feature not presented in \cite{bfw-bohr}.

\begin{theorem}\label{thm:B+C+A_i_discrete}
Let $G$ be a countable discrete abelian group and let $B, C \subset G$ have positive upper Banach density. Then for any partition $G = \bigcup_{i=1}^r A_i$, there is an $i \in \{1, \ldots, r\}$ such that $B + C + A_i$ contains a Bohr-$(k, \eta)$ set where $k, \eta$ depend only on $d^*(B), d^*(C)$ and $r$. 
\end{theorem}


We deduce Theorems \ref{th:main-density}, \ref{th:main-partition} and \ref{thm:B+C+A_i_discrete} from their counterparts for compact abelian groups (i.e.~Theorems \ref{th:ll_compact} and \ref{thm:B+C+A_i_compact}). However, the latter can be used as black boxes and the reader does not need to know their inner workings. The heavy lifting of this paper is done by \textit{correspondence principles}, which state that sumsets in discrete abelian groups can be modeled by sumsets in compact abelian groups. This strategy dates back at least to Furstenberg's correspondence principle \cite{Furstenberg}, used in his proof of Szemer\'edi's theorem. However, to accommodate the three different kinds of sumsets in our results, we need three different correspondence principles.  These are \cref{lem:correspondence_funny}, \cref{prop:correspondence_principle_bohr}, and \cref{prop:B+C-A}.

Our bounds for $k$ and $\eta$ in Theorems \ref{th:main-density}, \ref{th:main-partition} and \ref{thm:B+C+A_i_discrete} are transferred from and have the same quality as their compact analogues. Since the proof of \cref{th:ll_compact} (i) relies on a regularity lemma, the bounds in \cref{th:main-density} are of tower type. The proof of \cref{th:ll_compact}(ii) relies on the Hales-Jewett theorem, so the bounds in \cref{th:main-partition} are extremely poor (albeit still primitive recursive). As for \cref{thm:B+C+A_i_discrete}, we get more appealing bounds of the form $\eta = \Omega( d^*(B) d^*(C) r^{-1})$ and $k = O( d^*(B)^{-2} d^*(C)^{-2} r^2)$, though these may not be optimal (see Question \ref{q:optimal}).

\subsection{Main ideas of the proofs} Here we outline the obstacles to proving Theorems \ref{th:main-density}, \ref{th:main-partition} and \ref{thm:B+C+A_i_discrete} and our strategies for overcoming them.  We will use  notation and terminology defined in \cref{sec:background}. 

\textbf{\cref{th:main-density}}:
To prove the first theorem, we find a parameterized solution to the relation 
\begin{equation}\label{eqn:p1wInSumset}\phi_1(w) \in \phi_1(A) + \phi_2(A) + \phi_3(A).
\end{equation} 
For instance, $w$ will satisfy (\ref{eqn:p1wInSumset}) if
\begin{equation*}\label{eqn:allBelong}
    u + w - \phi_2(v), u + \phi_1(v), \text{ and } u \text{ all belong to } A \text{ for some } u, v \in G.
\end{equation*}
Then Furstenberg's correspondence principle is applied to show that the set of such $w$ contains the support of the multilinear ergodic average:
\begin{equation}
\label{eq:define_I_w}
    I(w):=UC-\lim_{g \in G} \int_X f\cdot T_{\phi_1(g)} f\cdot T_{w-\phi_2(g)} f\, d\mu
\end{equation}
where $(X, \mu, T)$ is an ergodic $G$-system and $f: X \to [0, 1]$ is a measurable function with $\int_X f \, d \mu = d^*(A)$. As shown in \cite{abb}, the Kronecker factor $(Z, m_Z, R)$ is characteristic for the average in \eqref{eq:define_I_w} and so
\begin{equation*}
\label{eq:I_WKronecker}
    I(w) = UC-\lim_{g \in G} \int_X \tilde{f}\cdot R_{\phi_1(g)} \tilde{f}\cdot R_{w-\phi_2(g)} \tilde{f}\, dm_Z,
\end{equation*}
where $\tilde{f}:Z\to [0,1]$ satisfies $\int \tilde{f}\, dm_Z =\int f\, d\mu$ (see Section \ref{sec:Folner_sequence} for the definition of $UC-\lim$). 
In order to utilize the corresponding result in compact groups \cite{ll}, we need to show that the homomorphisms $\phi_1, \phi_2, \phi_3$ induce homomorphisms $\tilde{\phi}_j$ on $Z$ satisfying $\tilde{\phi}_j \circ \tau = \tau \circ \phi_j$, where $\tau$ is a natural embedding of $G$ in $Z$. This is straightforward under the additional assumption that spectrum of $(X, \mu, T)$ (i.e.~the group of eigenvalues) is closed under each $\phi_j$.  However, the spectrum of $(X,\mu,T)$ will not, in general, be closed under the $\phi_j$.

To overcome this problem, we find an ergodic extension $(Y, \nu, S)$ of $(X, \mu, T)$ such that the spectrum of $(Y,\nu,S)$ contains a subgroup $\Gamma$ which extends the spectrum of $(X,\mu,T)$ and is invariant under each $\phi_j$.  After lifting $f$ to $Y$, the Kronecker factor $\mb Z$ of $\mb X$ can be viewed as a factor of $\mb Y$, and is still characteristic for the averages in (\ref{eq:define_I_w}).  Thus, any extension of $\mb Z$ in $\mb Y$ will also be characteristic for these averages.  The group rotation factor $\mb{K}$ of $\mb Y$ corresponding to $\Gamma$ is such an extension of $\mb Z$, and this allows us to transfer the Bohr sets obtained in \cite{ll} to $G$. The diagram below demonstrates the relations among $\mb{X}, \mb{Y}, \mb{Z}$ and $\mb{K}$ where $\mb{Y} \to \mb{X}$ means $\mb{Y}$ is an extension of $\mb{X}$.

\begin{center}
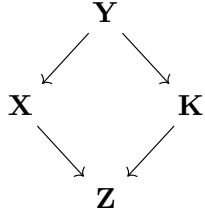

\begin{tikzcd}[column sep=small]
& \mb{Y} \arrow[ld] \arrow[rd] &\\
\mb{X} \arrow[rd] & & \mb{K} \arrow[ld] \\
& \mb{Z} &
\end{tikzcd}
\captionof{figure}{Relations among $X, Y, Z$ and $K$} 


\end{center}

\textbf{\cref{th:main-partition}}: In contrast to the sumset $\phi_1(A) + \phi_2(A) + \phi_3(A)$, a parametrized solution to $\phi_2(w) \in \phi_1(A) + \phi_2(A) - \phi_2(A)$ is 
\begin{equation}
\label{eq:parametrize_partition}
    \phi_2(v), u + w, u + \phi_1(v) \in A.
\end{equation}
The absence of the variable $u$ in the first function prohibits us from using Furstenberg's correspondence principle as we do in the Proof of \cref{th:main-density}. Instead we use (\cref{prop:correspondence_principle_bohr}), which models the relevant sumsets by convolutions on the Bohr compactification of $G$. This idea was used in \cite{bg} to express $A+A-A$ in terms of convolutions on a compact group. Parts of this process also already appeared in F{\o}lner's works \cite{folner1, folner2}.   

Specifically, we fix an invariant mean $\nu$ on $G$ with $d^*(A)=\nu(1_A)$, and observe that the difference set $A-A$ contains the support of the convolution $1_A*_\nu 1_{-A}(t):= \nu(1_A 1_{A + t})$.  This convolution is easily verified to be a positive definite function on $G$, which can therefore be represented as a Fourier transform of a positive measure $\sigma$ on $\widehat{G}$.  The continuous part of $\sigma$ can be ignored, allowing us to expand $1_A*_\nu 1_{-A}(t)$ as a Fourier series and express $A+A-A$ in terms of a convolution $h_A*h_A*h_{-A}$ on $bG$, the Bohr compactification of $G$.  

To study the more complicated expression $\phi_1(A) + \phi_2(A)-\phi_2(A)$, we need to investigate the relationship between $1_A*_\nu 1_{-A}$ and $1_{\phi_2(A)}*_{\nu} 1_{-\phi_2(A)}$. This investigation leads to the introduction of Radon-Nikodym densities $\rho_{A}^{\nu}, \rho_{\phi_2(A)}^{\nu}$ and their relationship in \cref{sec:rn}. After the required relationship is established, we put all ingredients together (\cref{prop:correspondence_principle_bohr}, \cref{cor:right_containment}) and use the compact counterpart in \cite{ll} to prove \cref{th:main-partition}.

\vspace{1em}

\textbf{\cref{thm:B+C+A_i_discrete}}:  This last theorem relies on two ingredients:
\begin{enumerate}
\item[(i)] an estimate for the rank and radius of a Bohr set in sumsets of the form $B+C+A_i$, where $B, C$ are subsets of a compact abelian group $K$ and $K= \bigcup_{i=1}^r A_i$.  We bound the rank and radius in terms of $m_K(B)$, $m_K(C)$, and $r$, using the pigeonhole principle and elementary estimates on Fourier coefficients.

\item[(ii)] a correspondence principle relating the expression $B+C+A_i$ in a discrete abelian group to an analogous expression in a compact abelian group.
\end{enumerate}

The two correspondence principles previously mentioned do not apply to the expression $B+C+A_i$; see \cref{remark:different_correspondence}. Instead, we use a result from \cite{griesmer-small-sum-pairs} which exhibits piecewise Bohr structure in $B + C$. This allows us to relate $B+C+A_i$ to a convolution $h_B*h_C*h_{A_i}$ on a compact group $K$, where each of these functions takes values in $[0,1]$, $\int h_B\, dm_K \geq d^*(B)$, $\int h_C\, dm_K\geq d^*(C)$, and $\sum_{i=1}^r h_{A_i}\geq 1_K$.

\begin{remark}
\label{remark:different_correspondence}
None of the three correspondence principles outlined above subsumes the others.  The sumset $\phi_1(A) + \phi_2(A) + \phi_3(A)$ with $\phi_1 + \phi_2 + \phi_3 = 0$ is translation invariant (replacing $A$ with a translate of $A$ does not affect this sumset) and so a straightforward application of Furstenberg's correspondence principle suffices. The second sumset $\phi_1(A) + \phi_2(A) - \phi_2(A)$ is no longer translation invariant and hence requires a different correspondence principle. Since the last sumset $B + C + A_i$ is neither translation invariant nor has the form $A + B - B$, we need yet another correspondence principle. Conversely, one cannot use the third principle for the first two sums since this principle does not retain the relations among the summands which are present in the fact that $\phi_1(A), \phi_2(A), \phi_3(A)$ are images of the same set $A$.
\end{remark}

\subsection{Outline of the article}
In \cref{sec:background}, we set up notation and present some basic facts about measure preserving systems, Bohr compactifications, Kronecker factors, etc. 
In Section \ref{sec:kronecker} we describe a general construction of homomorphisms from discrete groups into compact groups with dense image. This construction is used in the proofs of all of our results. Section \ref{sec:rn} is devoted to transferring functions on discrete groups to compact groups, an ingredient used in the proofs of Theorems \ref{th:main-partition} and \ref{thm:B+C+A_i_discrete}. After these preliminaries,  \cref{th:main-density} is proved in Sections \ref{sec:reducing_to_integrals} and \ref{sec:existence_bohr}, then \cref{th:main-partition} is proved in Sections \ref{sec:second_correspondence} and \ref{sec:bohr_sets_partition}. We prove the correspondence principle needed for \cref{thm:B+C+A_i_discrete} in \cref{sec:B+C+A_i_correspondence} and establish the theorem in \cref{sec:proof_B+C+A_i}. Lastly, we present some open questions in \cref{sec:open_question}.

\textbf{Acknowledgement.} We thank the anonymous referee for carefully reading the manuscript, pointing out some oversights, and providing many suggestions which help improve the presentation of the paper. The third author is partially supported by NSF Grant DMS-2246921.

\section{Background}
\label{sec:background}

\subsection{Notation and convention}

Throughout this paper, $G$ is a countable discrete abelian group, and $K$ is used to denote a compact Hausdorff abelian group. We use $m_K$ to denote the unique probability Haar measure on $K$. The set of all continuous functions on $K$ is denoted by $C(K)$.


For $r \in \N$, we use $[r]$ to denote $\{1, 2, \ldots, r\}$. By the support of a function $f$, denoted by $\supp f$, we mean $\{ x : f(x) \neq 0\}$.

\subsection{F{\o}lner sequences and uniform Ces\`aro averages}

\label{sec:Folner_sequence}

A sequence $\mb{F} = (F_N)_{N \in \N}$ of finite subsets of $G$ is a \textit{F{\o}lner sequence} if for all $g \in G$, 
\[
    \lim_{N \to \infty} \frac{|F_N \triangle (g + F_N)|}{|F_N|} = 0.
\]
Every countable abelian group admits a F{\o}lner sequence. This is due to the fact that all discrete abelian groups are amenable, and having a F{\o}lner sequence is one of the many equivalent definitions of amenability for countable discrete groups (see \cite{pier}).

If $\mb{F}$ is a F{\o}lner sequence and $A \subset G$, the \emph{upper density of $A$ with respect to $\mb{F}$} is 
\begin{equation*}
    \overline{d}_{\mb{F}}(A) := \limsup_{N \to \infty} \frac{|A \cap F_N|}{|F_N|}.
\end{equation*}
The \emph{upper Banach density of $A$} is
\begin{equation} \label{eq:density3}
    d^*(A) := \sup \{d_{\mb{F}}(A): \mb{F} \text{ is a F{\o}lner sequence}\}.
\end{equation}
(For a proof that the definitions \eqref{eq:density2} and \eqref{eq:density3} are equivalent, see \cite[Proposition A.6]{bjorklund-fish}.)

Let $u: G \to \C$ be a bounded sequence. We say $(u(g))_{g \in G}$ has a \emph{uniform Ces\`aro average} if for every F{\o}lner sequence $(F_N)_{N \in \N}$, the limit
\begin{equation*}
    \lim_{N \to \infty} \frac{1}{|F_N|} \sum_{n \in F_N} u(g)
\end{equation*}
exists and is independent of the choice of F{\o}lner sequence. In this case, we denote the common limit by $UC-\lim_{g \in G} u(g)$.

\subsection{Measure preserving systems}\label{sec:MPS}

A \emph{measure preserving $G$-system} (or \emph{$G$-system}) is a quadruple $\bX = (X, \calB, \mu, T)$ where $(X, \calB, \mu)$
is a probability space and $G$ acts on $X$ by transformations $T_g$ which preserve $\mu$; that is
\[
    \mu(T_g^{-1} A) = \mu(A)
\]
for all measurable $A \subset X$ and all $g \in G$. In this paper, all probability spaces underlying $G$-systems are assumed to be \textit{separable}, that is, $\calB$ is countably generated modulo null sets, or equivalently, $L^p(X, \calB, \mu)$ is separable for all $1\leq p < \infty$. In particular, if $X$ is a compact metric space, $\calB$ is its Borel $\sigma$-algebra and $\mu$ is any probability measure on $\calB$, then $(X, \calB, \mu)$ is separable. When there is no danger of confusion, we will suppress the $\sigma$-algebra $\mathcal B$ and write $(X, \mu, T)$ for a $G$-system.  We abbreviate $G$-systems with boldface letters: $\mb X=(X,\mu,T)$.  

The $G$-system $(X, \calB, \mu, T)$ is said to be \textit{ergodic} if $\mu(A \triangle T_g^{-1} A)=0$  for all $g \in G$ implies $\mu(A) = 0$ or $\mu(A)=1$.

If $f \in L^2(\mu)$ and $g \in G$, we write $T_g f$ for $f \circ T_g$. This defines an action of $G$ on $L^2(\mu)$ by unitary operators $T_g$. 

A $G$-system $\bY = (Y, \calD, \nu, S)$ together with a map $\pi: X\to Y$ defined for $\mu-$almost every $x \in X$ is a \emph{factor} of $\bX = (X, \calB, \mu, T)$  if $\pi_* \mu = \nu$ (i.e.~$\mu( \pi^{-1}(A)) = \nu(A)$ for all $A \in \calD$) and for all $g \in G$,
\[
    \pi(T_g x) = S_g \pi(x) \text{ for } \mu\text{-almost all } x \in X.
\]
The map $\pi$ is called a \emph{factor map}. 
The space $L^2(\nu)$ can be identified with the subspace of $L^2(\mu)$ consisting of functions of the form $h \circ \pi$ where $h \in L^2(\nu)$. We use $\E(\cdot|Y): L^2(\mu) \to L^2(\nu)$ to denote the corresponding orthogonal projection.   Later we abuse notation and write ``$\mb Y$ is a factor of $\mb X$'' instead of ``$(\mb Y,\pi)$ is a factor of $\mb X$.''

For a F{\o}lner sequence $(F_N)_{N \in \N}$ in $G$, functions $f_0, \ldots, f_k \in L^{\infty}(\mu)$, and sequences $s_1, \ldots, s_k: G \to G$, we say the factor $\mb Y$ is \emph{characteristic} for the average
\[
    I := \lim_{N \to \infty} \frac{1}{|F_N|} \sum_{g \in F_N} \int_X f_0 \cdot T_{s_1(g)} f_1 \cdots T_{s_k(g)} f_k \, d \mu
\]
if 
\[
    I = \lim_{N \to \infty} \frac{1}{|F_N|} \sum_{g \in F_N} \int_Y \tilde{f}_0 \cdot T_{s_1(g)} \tilde{f}_1 \cdots T_{s_k(g)} \tilde{f}_k \, d \nu
\]
where $\tilde{f}_i = \E(f_i|Y)$.

Let $\widehat{G}$ denote the Pontryagin dual of $G$, i.e.~the group of characters $\chi: G \to S^1$ with the operation of pointwise multiplication.
A character $\chi \in \widehat{G}$ called an \emph{eigenvalue} of $\bX$ if there exists a nonzero function $f \in L^2(\mu)$ such that $T_g f = \chi(g) f$ for all $g \in G$. The set of all eigenvalues for $\bX$ forms a subgroup of $\widehat{G}$, called the \textit{spectrum} of $\bX$ and denoted by $\calE(\bX)$. 
If $\bY$ is a factor of $\bX$, then $\calE(\bY)$ is a subgroup of $\calE(\bX)$.
If $\bX$ is ergodic, then all eigenspaces are one-dimensional and mutually orthogonal (for a proof, see \cite[Theorem 3.1]{walters}). Since $L^2(\mu)$ is separable, $\calE(\bX)$ is at most countable.

\subsection{Kronecker factors}

\label{sec:background_kronecker}

A \emph{group rotation $G$-system} is a $G$-system $\bK = (K, m_K, R)$ in which 
\begin{itemize}
    \item $K$ is a compact metrizable abelian group with Borel $\sigma$-algebra $\mathcal K$, probability Haar measure $m_K$, and
    
    \item there is a homomorphism $\tau: G \to K$ such $R_g(z) = z + \tau(g)$ for all $z \in K$ and $g \in G$.
\end{itemize} 
The group rotation $(K, m_K, R)$ is ergodic if and only if $\tau(G)$ is dense in $K$. In this case, $(K, m_K, R)$ is in fact \textit{uniquely ergodic}, i.e.~$m_K$ is the unique $R$-invariant probability measure on $K$ (for a proof, see \cite[Lemma 2.4]{abb}). Consequently, the sequence $(\tau(g))_{g \in G}$ is well-distributed in $K$, i.e.~for every continuous function $h \in C(K)$,
\begin{equation} \label{eq:well-distributed}
    UC-\lim_{g \in G} h(\tau(g)) = \int_K h \, d m_K.
\end{equation}

For an ergodic $G$-system $\bX$, its \emph{Kronecker factor} $\bK = (K,m_K,R)$ is a factor of $\bX$ with factor map $\pi: X\to K$ such that $L^2(m_K)$ is spanned by the eigenfunctions of $\bX$, meaning:
\begin{enumerate} 
\item[(i)] every eigenfunction $f\in L^2(\mu)$ is equal $\mu$-a.e.~to $\tilde{f}\circ \pi$ for some eigenfunction $\tilde{f}\in L^2(m_K)$, and 
\item[(ii)] the span of the eigenfunctions of $\bK$ is dense in $L^2(m_K)$.
\end{enumerate} 
It can be shown that $\bK$ is the largest factor of $\bX$ that is isomorphic to an ergodic group rotation $G$-system. More concretely, $\bK = (K, m_K, R)$ where $K= \widehat{\calE(\bX)}$ (see Lemma \ref{lem:rotations} (iii)).


Let $(X, \mu, T)$ be an ergodic $G$-system with Kronecker factor $(K, m_K,R)$ and $f_1, f_2, f_3\in L^{\infty}(X)$. It is shown in \cite[Theorem 3.1]{abb} that if $\phi, \psi: G \to G$ are homomorphisms such that $\phi(G)$, $\psi(G)$, and $(\psi-\phi)(G)$ each have finite index in $G$,
\begin{equation}
\label{eq:int_Kronecker}
    UC-\lim_{g \in G} \int_X f_1 \cdot T_{\phi(g)} f_2 \cdot T_{\psi(g)} f_3 \, d \mu
\end{equation}
exists and is equal to 
\[
    UC-\lim_{g \in G} \int_K \tilde{f}_1 \cdot R_{\phi(g)} \tilde{f}_2 \cdot R_{\psi(g)} \tilde{f}_3 \, d m_K
\]
where $\tilde{f}_i=\E(f_i|K)$ is projection of $f_i$ onto $L^2(m_K)$. In other words, the Kronecker factor is characteristic for the average in \eqref{eq:int_Kronecker}.

\subsection{Invariant means}

If $f \in \ell^{\infty}(G)$ and $t \in G$, define $f_t\in \ell^{\infty}(G)$ by $f_t(s) := f(s - t)$. An \emph{invariant mean} on $G$ is a positive linear functional $\nu: \ell^{\infty}(G) \to \C$ such that $\nu(1_G) = 1$ and $\nu(f_t) = \nu(f)$ for every $f \in \ell^{\infty}(G)$, $t \in G$. 

In the weak$^*$ topology on $\ell^{\infty}(G)^*$, 
the space $M(G)$ of invariant means forms a compact convex set. An invariant mean $\nu$ is said to be \textit{extremal}, or an \textit{extreme point}, if it cannot be written as a convex linear combination of two other invariant means. 

Bauer's maximum principle \cite[7.69]{ab} implies that if $C$ is a compact convex subset of a locally
convex Hausdorff space, then every real-valued continuous linear functional on $C$ has a maximizer that is an extreme point. Thus if $A \subset G$, there is an extremal invariant mean $\nu$ such that $d^*(A) = \nu(1_A)$. 

Let $H$ be a countable abelian group and $\phi: G \to H$ be a surjective homomorphism. For any invariant mean $\nu$ on $G$, the pushforward $\phi_* \nu$ is an invariant mean on $H$ and is defined by
\begin{equation*}
    \phi_*\nu(h) := \nu(h \circ \phi), 
\end{equation*}
for all $h \in \ell^{\infty}(H)$. Given $f \in \ell^{\infty}(G)$ and an invariant mean $\nu$, we sometimes write $\int_G f(t) \, d \nu(t)$ instead of $\nu(f)$. If $g \in \ell^{\infty}(G)$, we define the ``convolution'' of $f$ and $g$ with respect to $\nu$ by
\[
    f *_{\nu} g (t) := \int_G f(x) g(t - x) \, d \nu(x).
\]
In conventional notation, this could be written as $f*_{\nu} g:=\nu((g')_t f)$, where $g'(x):=g(-x)$. The following lemma is a special case of \cite[Proposition 2.1]{bjorklund-fish}.

\begin{lemma}
\label{lem:inv_mean_need_one_extremal}
If $\lambda$ is an extremal invariant mean on $G$ and $f, g \in \ell^{\infty}(G)$, then
\begin{equation} 
\label{eq:fubini}
    \iint_{G^2} f(t) g(t-s) \, d \lambda(t) d \mu(s) = \lambda(f) \lambda(g) 
\end{equation}
for every invariant mean $\mu$ on $G$.
\end{lemma}
For completeness we include a proof.

\begin{proof}
It suffices to prove \eqref{eq:fubini} for $0 \leq f \leq 1$. When $\lambda(f) = 0$ or 1, 
it is straightforward to check \eqref{eq:fubini}. Suppose $ \lambda(f) = \alpha \in (0,1)$. Define two invariant means $\eta$ and $\eta'$ by
\[
\eta(g) = \frac{1}{\alpha} \iint_{G^2} f(t) g(t-s) \, d \lambda(t) d \mu(s) \quad \textup{and} \quad \eta'(g) = \frac{1}{1-\alpha} \iint_{G^2} (1-f(t)) g(t-s) \, d \lambda(t) d \mu(s).
\]
Then it is easy to check that $\lambda(g) = \alpha \eta(g) + (1-\alpha) \eta'(g)$. Since $\lambda$ is extremal, we must have $\eta =\eta' = \lambda$, and we are done.
\end{proof}





\subsection{Bohr compactification}

The \emph{Bohr compactification} of $G$ is a compact abelian group $bG$, together with a homomorphism $\tau: G \to bG$ such that $\tau(G)$ is dense in $bG$ and every character $\chi \in \widehat{G}$ can be written as $\chi = \chi' \circ \tau$, where $\chi'$ is a continuous homomorphism from $bG$ to $S^1$. The homomorphism $\tau$ is universal with respect to homomorphisms into compact Hausdorff groups; that is if $K$ is another compact Hausdorff group and $\pi: G \to K$ is a homomorphism, then there is a unique continuous homomorphism $\tilde{\pi}: bG \to K$ such that $\pi = \tilde{\pi} \circ \tau$. The Bohr compactification also has a concrete description; it is the dual of $\widehat{G}$ where $\widehat{G}$ is given the discrete topology (see Section \ref{sec:kronecker}).

See \cite{rudin} for basic results on the Bohr compactification and \cite{bjorklund-fish} for a recent application to sumsets.

\subsection{Lemmas on Bohr sets}
We document two lemmas concerning Bohr sets for later use. Similar lemmas for compact abelian groups have been proved in \cite{ll}; the proofs for arbitrary abelian groups are identical and so we omit them. 




The first lemma states that the preimage of a Bohr set is a Bohr set. 

\begin{lemma} [{\cite[Lemma 2.9]{ll}}]

\label{lem:from_kronecker_up}
Let $G, H$ be abelian groups and $\tau: G \to H$ be a homomorphism. If $B$ is a Bohr-$(k, \eta)$ set in $H$, then $\tau^{-1}(B)$ is a Bohr-$(k, \eta)$ set in $G$.

\end{lemma}

The second lemma says that the image of a Bohr set under a homomorphism with finite index image is again a Bohr set.

\begin{lemma}[{\cite[Lemma 2.10]{ll}} and {\cite[Lemma 1.7]{griesmer-br}}] 
\label{lem:phiB_Bohr}
Let $G$ be an abelian group and $\phi: G \to G$ be an endomorphism with $[G: \phi(G)] < \infty$. If $B$ is a Bohr-$(k, \eta)$ set in $G$, then $\phi(B)$ is a Bohr-$(k', \eta')$ set in $G$ where $k', \eta'$ depend only on $k$, $\eta$, and $[G: \phi(G)]$.
\end{lemma}

\subsection{Almost periodic functions and null functions}

\label{sec:almost_per}

A function on $G$ of the form $g \mapsto \sum_{i=1}^k c_i \chi_i(g)$ where $c_i \in \C$ and $\chi_i \in \widehat{G}$ is called a \emph{trigonometric polynomial}.

An $f \in \ell^{\infty}(G)$ is called a \emph{(Bohr) almost periodic function} if it is a uniform limit of a sequence of trigonometric polynomials. Alternatively, $f$ is almost periodic if $f = h \circ \tau$ where $h$ is a continuous function on $bG$ and $\tau: G \to bG$ is the natural embedding. 
Given an almost periodic function $f$, a $\chi\in \widehat{G}$, and an invariant mean $\nu$ on $G$, we write $\hat{f}(\chi)$ for the \emph{Fourier coefficient} $\nu(f\overline{\chi})$ - it is easy to verify that for an almost periodic $f$, $\hat{f}(\chi)$ does not depend on the choice of $\nu$.

An $f \in \ell^{\infty}(G)$ is called a \emph{null function} if $\nu(|f|) = 0$ for every invariant mean $\nu$ on $G$. 

     


\section{Dense images of discrete groups in compact groups} 
\label{sec:kronecker} 





This section describes a general way to construct a homomorphism $\tau:G\to K$ from a discrete abelian group $G$ into a compact abelian group $K$.  It also provides sufficient conditions for an endomorphism $\phi$ of $G$ to induce an endomorphism $\tilde{\phi}$ of $K$. This framework provides a concrete description of the Bohr compactification of $G$ and of the Kronecker factor of an ergodic $G$-system. We start with the following.

\begin{lemma}\label{lem:phiStarProperties}
	Let $\Gamma$ be a locally compact abelian group and let $\phi:\Gamma\to \Gamma$ be a continuous endomorphism.  Define an endomorphism $\phi^*:\widehat{\Gamma}\to\widehat{\Gamma}$ by $\phi^*(\chi)=\chi \circ \phi$. Then
	\begin{enumerate}
		\item[(i)]   $\phi^*$ is continuous.
		
		\item[(ii)]			Under the canonical identification of $\widehat{\widehat{\Gamma}}$ with $\Gamma$, $(\phi^*)^*=\phi$.
	\end{enumerate}
	
\end{lemma}

\begin{proof}

	
	(i) By definition, $\widehat{\Gamma}$ is equipped with the topology of uniform convergence on compact subsets of $\Gamma$.  It therefore suffices to prove that if $(\chi_n)_{n\in I}$ is a net of elements of $\widehat{\Gamma}$ converging to $\chi\in \widehat{\Gamma}$ uniformly on compact subsets of $\Gamma$, then $(\chi_n\circ \phi)_{n\in I}$ converges to $\chi \circ \phi$ uniformly on compact subsets of $\Gamma$.  Continuity of $\phi$ implies $\phi(K)$ is compact for every compact $K\subset \Gamma$, so the assumption that $\chi_n\to \chi$ uniformly on every compact $K\subset \Gamma$ implies $\chi_n \to \chi$ uniformly on $\phi(K)$ for every compact $K\subset \Gamma$.  But this means $(\chi_n \circ \phi)_{n\in I}$ converges to $\chi\circ \phi$ uniformly on compact subsets of $\Gamma$, as desired.

	(ii) For $\gamma\in \Gamma$, define the evaluation map $e_\gamma(\chi) = \chi(\gamma)$ for any $\chi \in \widehat{\Gamma}$.  It suffices to prove that
	\[
	(\phi^*)^*(e_\gamma) = e_{\phi(\gamma)},
	\]
	meaning $(\phi^*)^*(e_\gamma)(\chi) = \chi(\phi(\gamma))$ for all $\chi\in\widehat{\Gamma}$.  To see this, note that $\chi\mapsto (\phi^*)^*(e_\gamma)(\chi)$ is defined by $e_{\gamma}(\phi^*(\chi))=e_{\gamma}(\chi\circ\phi).$
\end{proof}


We now apply \cref{lem:phiStarProperties} in the case where $\Gamma$ is a discrete group.

\begin{lemma}\label{lem:Embeddings} 
Let $\Lambda$ be a subgroup of $\widehat{G}$, viewed as a discrete group, so that  $\widehat{\Lambda}$ is compact. For $g \in G$, define the evaluation map $e_g (\chi) = \chi(g)$ for $\chi \in \widehat{G}$. Define a homomorphism $\tau: G\to \widehat{\Lambda}$ by $\tau(g)=e_g|_{\Lambda}$.  Then
	
	\begin{enumerate}  \item[(i)]   $\tau(G)$ is dense in $\widehat{\Lambda}$.
		
		\item[(ii)]  Suppose $\phi:G\to G$ is an endomorphism such that $\chi\circ \phi\in \Lambda$ for all $\chi\in \Lambda$.  Then there is a continuous endomorphism $\tilde{\phi}$ of $\widehat{\Lambda}$ such that $\tilde{\phi}\circ \tau=\tau\circ \phi$.  Furthermore, $[\widehat{\Lambda}:\tilde{\phi}(\widehat{\Lambda})]\leq [G:\phi(G)]$.
		
	\end{enumerate}
\end{lemma}

\begin{proof} (i) Let $\psi\in\widehat{\Lambda}$, let $F = \{\chi_1,\dots,\chi_d\} \subset \Lambda$ be finite, and $\varepsilon>0$.  We will show that there is a $g\in G$ such that $|\psi(\chi_j)-e_g(\chi_j)|<\varepsilon$ for all $\chi_j\in F$.  Consider the subgroup
	\[
	H:=\{(\chi_1(g),\dots,\chi_d(g)):g\in G\}\subset (S^1)^d.
	\]
	It suffices to prove that 
	\begin{equation}\label{eqn:vect}
	\vec{t}:=(\psi(\chi_1),\dots,\psi(\chi_d))\in \overline{H}.
	\end{equation}  
	Assume, to get a contradiction, that (\ref{eqn:vect}) is false.  Then there is a nontrivial character $\alpha\in \widehat{(S^{1})^d}$ which annihilates $\overline{H}$ but does not annihilate $\vec{t}$. Writing $\alpha(x_1,\dots,x_d)$ as $x_1^{n_1} \cdots x_d^{n_d}$, we have
	\begin{equation}
	\label{eqn:Relation}
	\chi_1(g)^{n_1}\cdots \chi_d(g)^{n_d} = 1 \qquad \text{for all } g\in G,
	\end{equation}
	but $\psi(\chi_1)^{n_1}\cdots \psi(\chi_d)^{n_d}\neq 1$.  Since $\psi$ is a character, the latter equation means \begin{equation}
	\label{eqn:NontrivialPsi}
	\psi(\chi_1^{n_1}\cdots\chi_d^{n_d})\neq 1.
	\end{equation}
	But (\ref{eqn:Relation}) means that $\chi_1^{n_1}\cdots \chi_d^{n_d}$ is trivial, contradicting (\ref{eqn:NontrivialPsi}). 
	
	\medskip
	
	\noindent (ii)  Define $\phi':\Lambda\to \Lambda$ by $\phi'(\chi)=\chi\circ \phi$.  Let $\tilde{\phi}:=(\phi')^*$ as in Lemma \ref{lem:phiStarProperties}, meaning that for $\psi\in\widehat{\Lambda}$, $\tilde{\phi}(\psi)=\psi\circ \phi'$. By Lemma \ref{lem:phiStarProperties}, $\tilde{\phi}$ is a continuous endomorphism.  To verify that $\tilde{\phi}\circ\tau=\tau\circ\phi$, fix $\chi\in\Lambda$, $g\in G$, and evaluate
	\[
	\tilde{\phi}(\tau(g))(\chi)=e_g(\phi'(\chi))=e_g(\chi\circ \phi)=\chi\circ\phi(g)=e_{\phi(g)}(\chi)=\tau(\phi(g))(\chi).\]
	Thus $\tilde{\phi}\circ\tau = \tau\circ \phi$.
	
	Now let $k=[G:\phi(G)]$ (assuming this index is finite), and let $t_j+\phi(G), j=1,\dots,k$ be coset representatives of $\phi(G)$.    The identity $\tilde{\phi}\circ\tau=\tau\circ \phi$ implies $\tilde{\phi}(\widehat{\Lambda})$ contains $\overline{\tau(\phi(G))}$.  The latter subgroup has index at most $k$, since the translates $\overline{\tau(t_j+\phi(G))} = \tau(t_j)+\overline{\tau(\phi(G))}$ are closed and cover a dense subset of $\widehat{\Lambda}$.  Thus $\tilde{\phi}(\widehat{\Lambda})$ also has index at most $k$.
	\end{proof}
	
	
	

It can be shown that all homomorphisms from $G$ into compact groups with dense images arise from the construction in \cref{lem:Embeddings}, though we do not need this fact. When $\Lambda = \widehat{G}$ with the discrete topology, $\widehat{\Lambda}$ is the Bohr compactification $bG$ of $G$, which is relevant in the proof of \cref{th:main-partition}. 


In the proofs of Theorems \ref{th:main-density} and \ref{thm:B+C+A_i_discrete}, we will focus on the case where $\Lambda$ is at most \textit{countable}. The relevance of countability is that, in this case, $\widehat{\Lambda}$ is compact and metrizable. Consequently, its Borel $\sigma$-algebra is separable (so the theory of factors applies).

The group $\wL$ being abelian, we can write its group operation additively. Equipped with its normalized Haar measure $m_{\wL}$, $\wL$ is naturally endowed with a group rotation via the $G$-action $R$ given by $R_g(z) := z + \tau(g)$ for all $z \in \wL$ and $g\in G$, where $\tau$ is defined in Lemma \ref{lem:Embeddings}. Since $\tau(G)$ is dense in $\wL$, this action is ergodic. We will now state some properties of these group rotations.

\begin{lemma} \label{lem:rotations}\
\begin{enumerate}
 \item [(i)] For all countable subgroups $\Lambda$ of $\widehat{G}$, we have $\calE(\wL, m_{\wL}, R) = \Lambda$.  Furthermore, all the eigenvectors of $R$ corresponding to the eigenvalue $\lambda \in \Lambda$ are constant multiples of $v_{\lambda}$, where $v_\lambda(x) = x(\lambda)$ for all $x \in \wL$.
 
        \item [(ii)] If $\Lambda_1 \leq \Lambda_2$ are countable subgroups of $\widehat{G}$, then the group rotation associated with $\widehat{\Lambda_1}$ is a factor of the group rotation associated with $\widehat{\Lambda_2}$.
        
    \item [(iii)] If $\bX=(X,\mu,T)$ is an ergodic $G$-system and $\Lambda = \calE(\bX)$, then $(\wL, m_{\wL}, R)$ is the Kronecker factor of $\bX$. 

\end{enumerate}
\end{lemma}
\begin{proof}
\noindent (i) For $\lambda \in \Lambda$ and $x \in \widehat{\Lambda}$, we have 
\[v_\lambda(x + \tau(g))= (x + e_g)(\lambda) = x(\lambda) \lambda(g) = \lambda(g) v_\lambda(x).
\]
This shows that $\lambda$ is an eigenvalue of $(\wL, m_{\wL}, R)$ and $v_\lambda$ is a corresponding eigenvector. 

Conversely, suppose $\chi \in \widehat{G}$ and there exists non-zero $f \in L^2(\wL)$ such that for all $g \in G$, $f(x + \tau(g)) = \chi(g) f(x)$ for almost all $x$, we need to show that $\chi \in \Lambda$. Since $f$ is not zero, there exists $\lambda \in \Lambda$ such that $\widehat{f}(\lambda) \neq 0$.
Computing the Fourier coefficients of both sides, we have
\[
\chi(g) \widehat{f}(\lambda) = e_g(\lambda) \widehat{f}(\lambda) = \lambda(g) \widehat{f}(\lambda)  \]
for any $g \in G$. Since $\widehat{f}(\lambda) \neq 0$, this implies that $\chi(g) = \lambda(g)$ for any $g \in G$. Therefore, $\chi = \lambda \in \Lambda$. Furthermore, this also shows that $f$ has exactly one non-zero Fourier coefficient and $f = \widehat{f}(\lambda) v_\lambda$.

\noindent (ii) Define $\pi : \wL_2 \to \wL_1$ by $\pi(x) = x|_{\Lambda_1}$ for all $x \in \wL_2$. Then $\pi$ is a surjective, continuous group homomorphism. By \cite[Lemma 2.7]{ll}, $\pi$ is measure-preserving.

Recall that the homomorphisms from $G$ to $\wL_1$ and $\wL_2$ are $\tau_1 (g) = e_g|_{\Lambda_1}$ and $\tau_2 (g) = e_g|_{\Lambda_2}$. It is clear that 
\[
\pi( x + \tau_2(g)) = \pi(x) + \tau_1(g),
\]
thus showing that $\pi$ is a factor map.

\noindent (iii) We assume (see Section \ref{sec:MPS}) that $L^2(\mu)$ is separable.  For each $\lambda \in \Lambda = \calE(\bX)$, there is an eigenvector $f_\lambda \in L^2(\bX)$ such that $T_g f_\lambda = \lambda(g) f_\lambda$ for any $g \in G$. Arguing similarly to \cite[Theorem 3.4]{walters}, we may assume that $|f_\lambda|=1$ and $f_{\lambda \xi} = f_\lambda f_\xi$ for any $\lambda, \xi \in \Lambda$. Defining $V(v_\lambda) = f_\lambda$ and extending $V$ linearly, we have an isometry $V: L^2(\wL) \rightarrow L^2(\bX)$ satisfying $V(fg) = V(f) V(g)$ for any $f, g \in L^2(\wL)$. By \cite[Theorem 2.4]{walters}, $V$ induces a homomorphism of measure algebras, and therefore a factor map $\bX \rightarrow \wL$. Since $\calE(\wL) = \Lambda$, part (ii) shows that $\wL$ is the largest group rotation that is a factor of $\bX$.
\end{proof}

\section{Radon-Nikodym densities} \label{sec:rn}



In this section we make no assumption on the countability (or uncountability) of $G$.  In particular, the lemmas here will apply when $G$ is an arbitrary discrete abelian group.

\subsection{Definition of Radon-Nikodym densities} Let $K$ be a compact abelian group and $\tau: G \to K$ be a homomorphism such that $\tau(G)$ is dense in $K$. We describe a way to transfer a function $f : G \to [0,1]$ to a function $\rho: K \to [0,1]$ with the aid of invariant means. This construction follows the proof of \cite[Lemma 2.5]{griesmer-dense-set}  (cf.~Section 4 of \cite{bg}); it will be used in the proofs of Theorems \ref{th:main-partition} and \ref{thm:B+C+A_i_discrete}.  

\begin{definition} \label{def:rn}
Let $f: G \rightarrow [0,1]$ and let $\nu$ be an invariant mean on $G$. 
The \textit{Radon-Nikodym density} associated with $f$ and $\nu$
is a Borel measurable function $\rho_{f}^{\nu}:K\to [0,1]$ satisfying 
\begin{equation}
\label{eqn:lambda_2}
    \nu((h \circ \tau) \cdot f) = \int_K h \cdot \rho_f^{\nu} \, d m_K. 
\end{equation} for every continuous $h:K\to \mathbb C$.  It is unique up to $m_K$-measure $0$.
\end{definition}
Thus $\rho_{f}^{\nu}$ depends on the compact group $K$ and the map $\tau$. When $f=1_A$ is the characteristic function of a subset of $G$, we write $\rho_{A}^{\nu}$ in place of $\rho_{1_A}^{\nu}$ to avoid nested subscripts.

Given an invariant mean $\nu$ on $G$, and $f:G\to [0,1]$ we will prove that there is a function $\rho_{f}^{\nu}$ satisfying Definition \ref{def:rn}.   We first observe the following.
\begin{lemma}
\label{lem:uniqueness-Bohr}
For all $h \in C(K)$, we have
    \begin{equation} \label{eq:density-obs1}
        \nu(h \circ \tau) = \int_{K} h \, dm_{K}. 
    \end{equation}
\end{lemma}
\begin{proof}
We define a linear functional $L$ on $C(K)$ by
\[
L(h) :=   \nu(h \circ \tau).
\]
 By the Riesz representation theorem, there exists a regular Borel probability measure $m$ on $K$ such that $L(h) = \int_{K} h \, dm$. On the other hand, for any $g \in G$, we have  
\begin{equation}\label{eqn:taugTranslate}
L(h_{\tau(g)}) = \nu((h \circ \tau)_g) = \nu(h \circ \tau) = L(h) 
\end{equation}
by translation invariance of $\nu$. Since the map $x \mapsto h_x$ from $K$ to $C(K)$ is continuous, and since $\tau(G)$ is dense in $K$, (\ref{eqn:taugTranslate}) implies $L(h_x)=L(h)$ for all $x \in K$. Hence $m$ is translation invariant. By uniqueness of the Haar measure, we have $m=m_{K}$ as desired. 
\end{proof}

Given $f:G\to [0,1]$, we define a linear functional $\Lambda_f^{\nu}: C(K) \to \R$ by
\begin{equation}
\label{eqn:Lambda_A}
    \Lambda_f^{\nu} (h) := \nu((h \circ \tau) \cdot f). 
\end{equation}
Clearly $\Lambda_f^{\nu}$ is a positive linear functional; thus by the Riesz representation theorem, there exists a regular Borel measure $m$ on $K$ such that
\begin{equation}\label{eqn:Lambdafnu}
    \Lambda_f^{\nu} (h) = \int_K h \, d m
\end{equation}
for all $h \in C(K)$. 

\begin{lemma}\label{lem:AbsCont}
The measure $m$ defined by (\ref{eqn:Lambdafnu}) is absolutely continuous with respect to the Haar probability measure $m_K$ on $K$, and in fact $m(B)\leq m_K(B)$ for all Borel sets $B\subset K$.
\end{lemma}

\begin{proof}
First, by \eqref{eq:density-obs1}, we have
\begin{equation} 
\label{eq:cont}
\int_K h \, d m = \nu((h \circ \tau) \cdot f) \leq \nu(h \circ \tau) = \int_K h \, d m_K 
\end{equation}
for any $h \in C(K)$.

Let $B$ be any Borel set in $K$. By regularity of $m$ and $m_K$, there is an open set $U$, a closed set $V$, such that $V \subset B \subset U$, $m( U \setminus V ) < \epsilon$ and $m_K( U \setminus V ) < \epsilon$. By Urysohn's lemma, there exists a continuous function $h : K \rightarrow [0,1]$ such that $h=1$ on $V$ and $h=0$ on $U^c$. Applying \eqref{eq:cont}, we have
\[
m(B) \leq m(V) + \epsilon \leq \int_{K} h \, dm + \epsilon \leq \int_{K} h \, dm_K + \epsilon 
\leq m_K(U) +  \epsilon \leq m_K(B) + 2 \epsilon.
\]
Since $\epsilon$ is arbitrary, this implies that $m(B) \leq m_K(B)$.
Therefore, $m$ is absolutely continuous with respect to $m_K$. 
\end{proof}

We now prove that, for each $f:G\to [0,1]$, there is a $\rho_{f}^{\nu}$ satisfying (\ref{eqn:lambda_2}). Given such an $f$, we consider the measure $m$ on $K$ defined above.  Since $m$ is absolutely continuous with respect to $m_K$, we may define $\rho_{f}^{\nu}$ to be the Radon-Nikodym derivative of $m$ with respect to $m_K$, meaning $\rho_{f}^{\nu}$ is the unique (up to $m_K$-measure $0$) function in $L^1(m_K)$ satisfying $\int h\, \rho_{f}^{\nu} \, dm_K = \int h\, dm$ for all $h\in C(K)$.  Then (\ref{eqn:lambda_2}) follows from (\ref{eqn:Lambda_A}) and (\ref{eqn:Lambdafnu}).  The inequality $0\leq \rho_{f}^{\nu}\leq 1$ $m_K$-a.e.~follows from the fact that $0\leq m(B)\leq m_K(B)$ for all Borel sets $B.$

\subsection{Properties of \texorpdfstring{$\rho_{A}^{\nu}$}{rho1A}}
We will now state some properties of $\rho^\nu _{f}$ when $f$ is the characteristic function of a set.  Recall that we write $\rho_{A}^{\nu}$ in place of $\rho_{1_A}^{\nu}$.

\begin{lemma}
\label{lem:supp_rho_A}
Let $A \subset G$ and let $\nu$ be an invariant mean on $G$. Then 
\begin{enumerate}
    \item[(i)] $\int_K \rho_{A}^{\nu} \, d m_K = \nu(1_A)$,
    \item[(ii)] $\rho_{A}^{\nu}$ is supported on $\overline{\tau(A)}$, that is, $\rho_{A}^{\nu} = 0$ $m_K$-a.e. on $K \setminus \overline{\tau(A)}$.
\end{enumerate}
\end{lemma}

\begin{proof}
The first claim follows from the definition of $\rho_{A}^{\nu}$. 
For the second claim, let $h: K \to \R_{\geq 0}$ be any continuous function that is supported on $K \setminus \overline{\tau(A)}$. 
If $g \in A$, then $\tau(g) \in \tau(A)$ will not be in the support of $h$. In other words, $h \circ \tau \cdot 1_A(g) = 0$ for all $g \in G$, and so
\begin{equation} 
\label{eq:urysohn}
    \int_K h \cdot \rho_{A}^{\nu} \, d m_K = \nu((h \circ \tau) \cdot 1_{A}) = 0. 
\end{equation}
Suppose for a contradiction that there exists a Borel set $V \subset K \setminus \overline{\tau(A)}$ with $m_K(V) >0$ such that $\rho_{A}^{\nu} >0$ on $V$. Since $m_K$ is regular, we may assume that $V$ is closed. By Urysohn's lemma, there is a continuous function $h: K \rightarrow [0,1]$ that is equal to $1$ on $V$ and $0$ on $\overline{\tau(A)}$. Then \eqref{eq:urysohn} implies that $\int_V \rho_{A}^{\nu} dm_K = 0$, a contradiction.
\end{proof}

\begin{lemma}
\label{cor:sum_rho_1}
Let $G = \bigcup_{i=1}^r A_i$ be a partition of $G$ and let $\nu$ be an invariant mean on $G$. Then 
\begin{equation*}
    \sum_{i=1}^r \rho_{A_i}^{\nu} (x) = 1
\end{equation*}
for $m_K$-almost every $x$.
\end{lemma}
\begin{proof}
Since $\sum_{i = 1}^r 1_{A_i} = 1$, for any $h \in C(K)$, 
\[
    \int_K h \left(\sum_{i=1}^r \rho_{A_i}^{\nu} \right) \, d m_K = \sum_{i=1}^r \nu(h \circ \tau \cdot  1_{A_i}) = \nu(h \circ \tau) = \int_K h \, d m_K
\]
where the last equality comes from \cref{lem:uniqueness-Bohr}. Since $C(K)$ is dense in $L^1(m_K)$, this implies that $\sum_{i=1}^r \rho_{A_i}^{\nu} = 1$ almost everywhere.
\end{proof}

\subsection{Relation between \texorpdfstring{$\rho_{A}$}{rho1A} and \texorpdfstring{$\rho_{\phi(A)}$}{rho1phiA}}

Let $G=A_1\cup \cdots \cup A_r$.  Our proof of Theorem \ref{th:main-partition} relies on a correspondence principle relating $\phi_1(A_i) + \phi_2(A_i)-\phi_2(A_i)$ to a convolution of the form $1_{\tilde{\phi}_1(B_i)}*1_{\tilde{\phi}_2(B_i)}*1_{\tilde{\phi}_2(-B_i)}$ on a compact abelian group $K$.  To prove such a correspondence principle, 
we need Lemma \ref{lem:SurjectiveRho} and Corollary \ref{cor:right_containment}, which specify the relationship between the Radon-Nikodym  densities of $1_A$ and $1_{\phi(A)}$.  In order to make the relevant issues apparent, the next lemma takes place in slightly greater generality than we need for our application.




\begin{lemma}\label{lem:SurjectiveRho}
Let $G$ and $H$ be discrete abelian groups and let $\phi:G\to H$ be a surjective homomorphism.
Let $K_1$, $K_2$ be compact abelian groups and $\tau_1: G\to K_1$, $\tau_2:H\to K_2$ be homomorphisms with dense images. Suppose $\tilde{\phi}:K_1\to K_2$ is a continuous surjective homomorphism such that 
\begin{enumerate}[label=(\roman*)]
    \item $\tilde{\phi} \circ\tau_1 = \tau_2\circ \phi$, and
    \item\label{item:AnnoyingClosure}  for all $\chi\in \widehat{K}_1$, if there is a $\psi\in \widehat{H}$ such that $\chi\circ \tau_1=\psi\circ \phi$, then there is a $\chi'\in \widehat{K}_2$ such that $\psi = \chi'\circ \tau_2$ (see Diagram \eqref{eq:tik_ii}).
\end{enumerate}  
Let $f: H\to [0,1]$ and let $\nu$ be an invariant mean on $G$. Let $\rho_{f \circ \phi}^{\nu}: K_1 \to [0,1]$ and $\rho_f^{\phi_* \nu}: K_2 \to [0,1]$ be the associated Radon-Nikodym densities as in \cref{def:rn}. Then 
\begin{equation*}
    \rho_{f\circ\phi}^\nu = (\rho_f^{\phi_*\nu})\circ \tilde{\phi}
\end{equation*}
$m_{K_1}$-almost everywhere.
\vspace{-1em}
\begin{center}
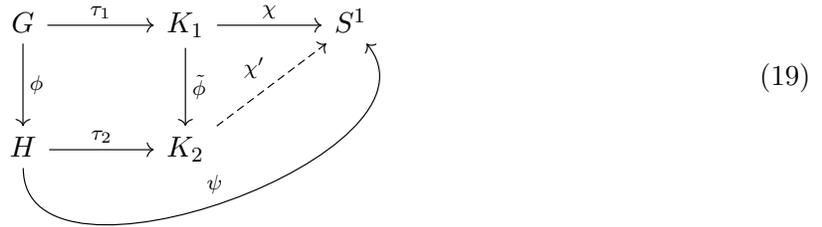

\begin{equation}\label{eq:tik_ii}
\begin{tikzcd}[column sep=large, row sep = large]
G \arrow[r, "\tau_1"] \arrow[d, "\phi"] & K_1 \arrow[r, "\chi"] \arrow[d, "\tilde{\phi}"] & S^1 \\
H \arrow[r, "\tau_2"] \arrow[urr, bend right = 110, "\psi"] & K_2 \arrow[ru, dashed, "\chi'"] & 
\end{tikzcd}
\end{equation}
\vspace{-4em}
\captionof{figure}{Illustration of \ref{item:AnnoyingClosure}}
\end{center}

\end{lemma}

\begin{remark}\
\begin{itemize}
    \item The surjectivity of $\phi$ is required for $\phi_*\nu$ to be an invariant mean on $H$, and thus for $\rho_f^{\phi_*\nu}$ to be defined on $K_2$.
    
    \item The assumption \ref{item:AnnoyingClosure} is satisfied by the groups we use in the proof of Theorem \ref{th:main-partition}; namely $K_1$ will be the Bohr compactification of $G$, $K_2$ will be $\tilde{\phi}(K_1)$, which will coincide with the Bohr compactification $bH$ of $H$, and $\tau_2:H\to K_2$ will be the usual embedding of $H$ into $bH$.
\end{itemize}

\end{remark}

\begin{proof}

We will prove that
\begin{equation}\label{eqn:FourierEqual}
\widehat{\rho_{f\circ\phi}^\nu}=\widehat{(\rho_f^{\phi_* \nu})\circ \tilde{\phi}}.
\end{equation} We first identify some characters of $G$ which are orthogonal to $f\circ \phi$.

\begin{claim}\label{claim:PsiForm}
Let $\psi\in \widehat{G}$.  Then $\nu((f\circ \phi) \cdot \overline{\psi})=0$ unless $\psi=\psi'\circ \phi$ for some $\psi'\in \widehat{H}$.

Similarly, if $\chi\in \widehat{K}_1$, and $h\in L^2(m_{K_2})$, then $\widehat{h\circ \tilde{\phi}}(\chi)=0$ unless $\chi=\chi'\circ \tilde{\phi}$ for some $\chi'\in\widehat{K}_2$.
\end{claim}
  To see this, assume $\psi\in \widehat{G}$ does not have  the form $\psi'\circ \phi$ for some $\psi'\in \widehat{H}$.  Then there is a $g\in \ker \phi$ such that $\psi(g)\neq 1$.\footnote{Supposing $\psi(g)=1$ for all $g\in \ker{\phi}$, we define a character $\psi'$ on $H$ by $\psi'(\phi(g))=\psi(g)$.  This is well defined, since $\phi(g)=\phi(g')$ implies $\psi(g)=\psi(g')$.  To check that $\psi'(h+h')=\psi'(h)\psi'(h')$, choose $g,g'$ so that $\phi(g)=h$ and $\phi(g')=h'$, and evaluate $\psi'(h+h')$ as $\phi(g+g')=\phi(g)\phi(g')=\psi'(\phi(g))\psi'(\phi(h))$.} We then have
\begin{align*}
    \nu((f\circ \phi)\cdot \overline{\psi}) &= \nu\bigl(((f\circ \phi)\cdot \overline{\psi})_{g}\bigr)\\
    &= \nu\bigl((f\circ \phi)\cdot (\overline{\psi})_{g}\bigr)\\
    &= \psi(g) \nu\bigl((f\circ \phi)\cdot \overline{\psi}\bigr). 
\end{align*}
So $\nu((f\circ \phi)\cdot \overline{\psi}) = \psi(g)\nu((f\circ \phi)\cdot \overline{\psi})$,
which means $\widehat{f\circ\phi}(\psi)=0$, since $\psi(g)\neq 1$.  This proves the first statement in the claim, and the second statement is proved similarly.

\begin{claim}\label{claim:ChiForm}
Let $\chi\in \widehat{K}_1$. Then $\widehat{\rho^{\nu}_{f\circ \phi}}(\chi) = 0$ unless $\chi=\chi'\circ \tilde{\phi}$ for some $\chi'\in \widehat{K}_2$.
\end{claim}
To prove this claim, let $\chi\in \widehat{K}_1$.  Then
\[
    \widehat{\rho_{f\circ\phi}^\nu}(\chi)  = \int_{K_1} \rho_{f\circ\phi}^\nu\, \overline{\chi}\, dm_{K_1} 
         = \nu \left( \left( f\circ \phi \right) \cdot \left( \overline{\chi}\circ \tau_1 \right) \right).
\]
By Claim \ref{claim:PsiForm}, the above evaluates to $0$ unless $\chi\circ \tau_1=\psi \circ \phi$ for some $\psi\in \widehat{H}$.  Choosing such a $\psi$, we have
\[
    \widehat{\rho_{f\circ\phi}^\nu}(\chi) = \nu((f\circ \phi)\cdot (\overline{\psi} \circ \phi)) = \phi_*\nu(f\overline{\psi}).
\]
By assumption \ref{item:AnnoyingClosure}, we may write $\psi$ as $\chi'\circ \tau_2$ for some $\chi'\in \widehat{K}_2$.  Then $\chi\circ \tau_1 = (\chi'\circ \tau_2)\circ \phi = \chi'\circ  \tilde{\phi}\circ \tau_1$.  So $\chi \circ \tau_1 = \chi'\circ \tilde{\phi} \circ \tau_1$.  The denseness of $\tau_1(G)$ in $K_1$ and continuity of $\chi$ then implies $\chi = \chi'\circ \tilde{\phi}$.  This shows that $\widehat{\rho^{\nu}_{f\circ \phi}}(\chi) = 0$ unless $\chi=\chi'\circ \tilde{\phi}$ for some $\chi'\in \widehat{K}_2$.

We now prove equation (\ref{eqn:FourierEqual}).

\noindent \textbf{Case 1:} $\chi=\chi'\circ \tilde{\phi}$ for some $\chi'\in \widehat{K}_2$. Then
\begin{align*}
    \widehat{\rho_{f\circ\phi}^\nu}(\chi) & = \int_{K_1} \rho_{f\circ\phi}^\nu\, \overline{\chi}\, dm_{K_1} \\
    & = \nu \left( \left( f\circ \phi \right) \cdot \left( \overline{\chi}\circ \tau_1 \right) \right) && \text{by definition of } \rho_{f\circ \phi}^{\nu}\\
    & = \nu\left( (f\circ \phi) \cdot  \left( \overline{\chi'\circ \tilde{\phi}}\circ \tau_1 \right) \right)\\
    &= \nu\left( \left( f\circ \phi \right) \cdot \left(  \overline{\chi'\circ \tau_2}\circ \phi\right) \right)\\
    &= \phi_* \nu\left(f\, \cdot  \overline{\chi'\circ \tau_2}\right)\\
    &= \int_{K_2} \rho_f^{\phi_* \nu} \overline{\chi'}\, dm_{K_2}\\
    &= \int_{K_1} \left( \rho_f^{\phi_*\nu} \circ \tilde{\phi} \right) \cdot \left( \overline{\chi'}\circ \tilde{\phi} \right)\, dm_{K_1}\\
    &= \widehat{ \rho_f^{\phi_*\nu} \circ \tilde{\phi} }(\chi).
\end{align*}

\noindent \textbf{Case 2:} $\chi\neq \chi'\circ \tilde{\phi}$ for all $\chi'\in \widehat{K}_2$.  In this case, Claim \ref{claim:PsiForm} implies  $\widehat{(\rho_f^{\phi_* \nu})\circ \tilde{\phi}}(\chi)=0$ and Claim \ref{claim:ChiForm} implies $\widehat{\rho_{f\circ\phi}^\nu}(\chi)=0$. 
\end{proof}

\begin{corollary}
\label{cor:right_containment}
Let $G$ be a discrete abelian group, $\nu$  an invariant mean on $G$ and $\phi: G \to G$  an endomorphism. 
Let $K$ be a compact abelian group, $\tau: G \to K$ a homomorphism with dense image, and $\tilde{\phi}: K \to K$ an endomorphism such that $\tilde{\phi} \circ \tau = \tau \circ \phi$. Assume further that for all $\chi\in \widehat{K}$, if there is a $\psi\in \widehat{G}$ such that $\chi\circ \tau=\psi\circ \phi$, then there is a $\chi'\in \widehat{K}$ such that $\psi = \chi'\circ \tau$. Let $H = \phi(G)$, $A\subset G$, and let $\rho_{A}^{\nu}: K \to [0, 1]$ and $\rho_{\phi(A)}^{\phi_* \nu}: \tilde{\phi}(K) \to [0, 1]$ be the associated Radon-Nikodym densities. 
Then
\begin{equation*}
    0 \leq \rho_{A}^{\nu} \leq \rho_{\phi(A)}^{\phi_* \nu} \circ \tilde{\phi}
\end{equation*}
$m_{K}$-almost everywhere.
\end{corollary}

\begin{proof}
Applying \cref{lem:SurjectiveRho} for $H = \phi(G)$ and $f = 1_{\phi(A)}: H \to [0,1]$, we get
\begin{equation*}
\rho_{1_{\phi(A)} \circ \phi}^{\nu} = \rho_{1_{\phi(A)}}^{\phi_* \nu} \circ \tilde{\phi}.
\end{equation*}
Since $1_{\phi(A)} \circ \phi = 1_{\phi^{-1}(\phi(A))} \geq 1_A$, we have 
\begin{equation*}
    \rho_{1_{\phi(A)}\circ \phi}^{\nu} \geq \rho_{A}^{\nu}.
\end{equation*}
It follows that $    \rho_{1_A}^{\nu} \leq \rho_{1_{\phi(A)}}^{\phi_* \nu} \circ \tilde{\phi}$, meaning
\begin{equation*}
    \rho_{A}^{\nu} \leq \rho_{\phi(A)}^{\phi_*\nu}\circ \tilde{\phi}. \qedhere
\end{equation*}

\end{proof}

\section{Reducing correlation sequences to integrals in compact groups}
\label{sec:reducing_to_integrals}

The goal of this section is to show that certain averages for ergodic $G$-systems can be reduced to double integrals on a compact group. \cref{lem:IKronecker} establishes this for group rotations on a compact abelian group $K$, as long as some endomorphisms on $G$ can be extended to all of $K$.

\begin{lemma}\label{lem:IKronecker}
Let $K$ be a compact abelian group  and let $\tau:G\to K$ be a homomorphism with dense image. Let $\phi_1, \phi_2, \phi_3: G\to G$ be endomorphisms. Suppose there are continuous endomorphisms $\tilde{\phi}_i: K \to K$ such that $\tilde{\phi_i}\circ \tau = \tau \circ \phi_i$ for $1 \leq i \leq 3$. Then for all  bounded measurable  $f_1, f_2, f_3 :K\to \mathbb C$, we have
\begin{align*}
    I(\vec{f}, \vec{\phi}) &:= UC-\lim_{g \in G} \int_K f_1(z+\tau(\phi_1(g)))f_2(z+\tau(\phi_2(g)))f_3(z+\tau(\phi_3(g)))\, dm_K(z) \\
    &= \iint_{K^2} f_1(z+\tilde{\phi}_1(t))f_2(z+\tilde{\phi}_2(t))f_3(z+\tilde{\phi}_3(t))\, dm_K(z)\, dm_K(t).
\end{align*}
\end{lemma}

\begin{proof}
	Since $I(\vec{f},\vec{\phi})$ is continuous in $f_i$ (with respect to the $L^2(m_K)$-norm) and multilinear in $f_i$, it suffices to prove the identity when each $f_i$ is a character $\chi_i$ of $K$.  In this case we have
	\begin{align*}
	I(\chi_1,\chi_2,\chi_3,\vec{\phi})&= UC-\lim_{g \in G}\int_K \chi_1\chi_2\chi_3(z)\prod_{i=1}^3 \chi_i(\tau(\phi_i(g)))\, dm_K(z)\\
	&= UC-\lim_{g \in G} \int_K \chi_1 \chi_2 \chi_3(z) \prod_{i=1}^3 \chi_i \circ \tilde{\phi}_i(\tau(g))\, dm_K(z).
	\end{align*}
By \eqref{eq:well-distributed}, we have
\begin{align*}
    I(\chi_1, \chi_2, \chi_3, \vec{\phi}) &= \iint_{K^2} \chi_1 \chi_2 \chi_3(z) \prod_{i=1}^3 \chi_i \circ \tilde{\phi}_i(t) \, d m_K(z) d m_K(t) \\
    &= \iint_{K^2} \prod_{i=1}^3 \chi_i(z + \tilde{\phi}_i(t)) \, d m_K(z) d m_K(t),
\end{align*}
and this finishes our proof.
\end{proof}

The next proposition deals with a general ergodic $G$-system $\bX$. The compact group in question will be an extension $K$ of the group $Z$ underlying Kronecker factor of $\bX$, constructed to be invariant under the corresponding $\tilde{\phi}_i$, as required by Lemma \ref{lem:IKronecker}.

\begin{proposition} \label{prop:reduce_to_integral}
Given an ergodic measure preserving $G$-system $\mb X=(X,\mu, T)$ and $f:X\to [0,1]$, define $I:G\to \mathbb R_{\geq 0}$ by
\[
I(w):= UC-\lim_{g \in G} \int_X f \cdot T_{\phi_3(g)}f \cdot T_{w-\phi_2(g)} f\, d\mu,
\]
where $\phi_2, \phi_3: G\to G$ are endomorphisms such that $\phi_2, \phi_3, \phi_2 + \phi_3$ have finite index images in $G$.

	Then there are a compact abelian group $K$, a homomorphism $\tau: G\to K$ with dense image, endomorphisms $\tilde{\phi}_2, \tilde{\phi}_3:K\to K$ and $\tilde{f}:K\to [0,1]$ with $\int_K \tilde{f}\, dm_K = \int_X f\, d\mu$ such that for all $w \in G$,
	\begin{equation}\label{eqn:NewI}
	I(w) = \iint_{K^2} \tilde{f}(z)\tilde{f}(z+\tilde{\phi}_3(t))\tilde{f}(z+\tau(w)-\tilde{\phi_2}(t))\, dm_K(z)\, dm_K(t).
	\end{equation}
	Furthermore, $[K:\tilde{\phi}_i(K)] \leq [G:\phi_i(G)]$ for each $i \in \{2, 3\}$ and $[K: (\tilde{\phi}_2 + \tilde{\phi}_3)(K)] \leq [G: (\phi_2 + \phi_3)(G)]$.
\end{proposition}

\begin{proof}
	Let $\phi_1 = -\phi_2 - \phi_3$. We first prove the special case of the lemma where $\mathcal{E}(\bX)$ is invariant under each $\phi_i$, meaning that for all eigenvalues $\lambda\in \calE(\bX)$ and $i \in \{1, 2,3\}$, we have $\lambda \circ \phi_i\in \mathcal{E}(\bX)$. In this case, the conclusion was also observed in \cite[Remark 3.2]{abb}. By \cite[Section 3]{abb}, the Kronecker factor $(Z, m_Z, R)$ of $(X, \mu, T)$ is characteristic for the average defining $I(w).$ Let $\tau: G \to Z$ be the canonical projection. We can therefore replace $f$ with $\tilde{f} := \E(f|Z)$ without changing $I(w)$:
	\begin{align}
	I(w)&= UC-\lim_{g \in G} \int_Z \tilde{f} \cdot  R_{\phi_3(g)}\tilde{f} \cdot R_{w-\phi_2(g)}\tilde{f}\, dm_Z \nonumber\\
	&= UC-\lim_{g \in G} \int_Z \tilde{f}(z)\tilde{f}(z+\tau(\phi_3(g)))\tilde{f}(z+\tau(w-\phi_2(g)))\, dm_Z(z). \label{eq:IwEg}
	\end{align}
	
In view of Lemma \ref{lem:Embeddings}, let $\tilde{\phi}_i:Z\to Z$ be continuous endomorphisms satisfying $\tau\circ \phi_i=\tilde{\phi}_i\circ \tau$.  Applying this identity to \eqref{eq:IwEg}, we have
	\[
	I(w) = UC-\lim_{g \in G} \int_Z \tilde{f}(z)\tilde{f}(z+\tilde{\phi}_3(\tau(g)))\tilde{f}(z+\tau(w)-\tilde{\phi}_2(\tau(g))) \, dm_Z(z).
	\]
	By Lemma \ref{lem:IKronecker}, we can rewrite the previous line as
	\[
	I(w) =  \iint_{Z^2} \tilde{f}(z)\tilde{f}(z+\tilde{\phi}_3(t))\tilde{f}(z+\tau(w)-\tilde{\phi}_2(t))) \, dm_Z(z) \, dm_Z(t).
	\]
	Taking $K = Z$, we prove the proposition in this special case.
	
	For the general case, let $\Lambda$ be the smallest subgroup of $\widehat{G}$ that contains $\calE(\bX)$ and is closed under each $\phi_i^*$. Since $\calE(\bX)$ is countable, it is easy to see that $\Lambda$ is countable. Let $\bK = (\wL, m_{\wL}, R)$ be the group rotation on $\wL$ described in \cref{lem:rotations}. By part (i) of \cref{lem:rotations}, we have $\calE(\bK) = \Lambda$. Since $\calE(\bZ) = \calE(\bX) \subset \Lambda$, part (ii) of \cref{lem:rotations} implies that $\bZ$ is a factor of $\bK$.
	
We now fix an ergodic $G$-system $\mb{Y} = (Y, \nu, S)$ that is a common extension of $\mb{X}$ and $\mb{K}$. For example, we can take $\mb{Y} = (X \times K, \nu, T \times R)$ to be an ergodic joining of $\mb X$ and $\mb K$.  (For  details about joinings and the existence of ergodic joinings, see Glasner \cite[Section 6]{glasner} or de la Rue \cite[Section 3.1]{delaRue}.) 

	
	Writing $\pi: Y\to X$ for the factor map, we define $f': Y \to \mathbb [0,1]$ to satisfy $f':= f\circ \pi$ and  
	\[
	I'(w):= UC-\lim_{g \in G} \int_Y f' \cdot S_{\phi_3(g)}f' \cdot S_{w-\phi_2(g)} f' \, d\nu.
	\] 
Since $f'$ is a lift from $f$ on $X$, it is obvious that $I' = I$ and the Kronecker factor $\bZ$ of $\bX$ is characteristic for the averages $I'(w)$. Thus any factor of $\bY$ between $\bY$ and $\bZ$ is also characteristic for $I'(w)$. In particular, $\bK$ is characteristic for $I'(w)$.
Now applying an argument similar to the first part of the proof to the factor $\bK$ of $\bY$ and the function $f'$, we obtain the compact group $K = \wL$, the function $\tilde{f} = \E(f'| K)$, and endomorphisms $\tilde{\phi}_i$ satisfying (\ref{eqn:NewI}). Finally, we have $[K:\tilde{\phi}_i(K)] \leq [G:\phi_i(G)]$ for each $i \in \{1, 2, 3\}$ by \cref{lem:Embeddings} (ii).
\end{proof}

\section{First correspondence principle and Bohr sets in \texorpdfstring{$\phi_1(A) + \phi_2(A) + \phi_3(A)$}{phi1(A) + phi2(A) + phi3(A)}}
\label{sec:existence_bohr}

\begin{proposition}
\label{lem:support_bohr_countable}
Let $G$ be a countable abelian group. Let $\phi_1, \phi_2, \phi_3: G \to G$ be commuting endomorphisms with finite index images such that $\phi_1 + \phi_2 + \phi_3 = 0$. Let $(X, \mu, T)$ be an ergodic $G$-system and $f: X \to [0, 1]$ with $\int_X f = \delta > 0$. Define the function $I: G \to [0,1]$ by
\begin{equation*}
   I(w):= UC-\lim_{g \in G} \int_X f \cdot T_{\phi_3(g)} f \cdot T_{w-\phi_2(g)} f\, d\mu.
\end{equation*}
Then $\supp(I)$ contains a Bohr-$(k, \eta)$ set where $k, \eta$ depend only on $\delta$ and the indices of $\phi_i(G)$ in $G$.
\end{proposition}
\begin{proof}
By \cref{prop:reduce_to_integral}, there exist a compact abelian group $K$ with Haar measure $m_K$, a homomorphism $\tau: G\to K$ with dense image, and endomorphisms $\tilde{\phi}_i:K\to K$, and $\tilde{f}:K\to [0,1]$ with $\int_K \tilde{f}\, dm_K=\int_X f\, d\mu = \delta$ such that
	\begin{equation*}
	I(w) = \iint_{K^2} \tilde{f}(z)\tilde{f}(z+\tilde{\phi}_3(t))\tilde{f}(z+\tau(w)-\tilde{\phi}_2(t))\, dm_K(z)\, dm_K(t).
	\end{equation*}
	Furthermore, $[K:\tilde{\phi}_i(K)] \leq [G:\phi_i(G)]$ for each $i$. Now define $I': K \to [0, 1]$ by
	\begin{equation*}
	    I'(\widetilde{w}) := \iint_{K^2} \tilde{f}(z)\tilde{f}(z+\tilde{\phi}_3(t))\tilde{f}(z+ \widetilde{w} -\tilde{\phi}_2(t))\, dm_K(z)\, dm_K(t).
	\end{equation*}
	By change of variable $z \mapsto z + \tilde{\phi}_2(t)$ and using $\phi_2 + \phi_3 = - \phi_1$, we obtain
	\begin{equation*}
	    I'(\widetilde{w}) = \iint_{K^2} \tilde{f}(z + \tilde{\phi}_2(t) )\tilde{f}(z - \tilde{\phi}_1(t))\tilde{f}(z+ \widetilde{w})\, dm_K(z)\, dm_K(t).
	\end{equation*}
	Applying \cite[Proposition 4.3]{ll}, it follows that $\supp(I')$ contains a Bohr-$(k, \eta)$ set $B$ in $K$ where $k, \eta$ depends only on $\delta$ and the indices $[K:\tilde{\phi}_i(K)]$. It is easy to see that $\supp(I)$ contains $\tau^{-1}(B)$. Moreover,  \cref{lem:from_kronecker_up} implies that $\tau^{-1}(B)$ contains a Bohr-$(k, \eta)$ set in $G$, completing the proof.
\end{proof}



\begin{proposition}[First correspondence principle]
\label{lem:correspondence_funny}
Let $G$ be a countable abelian group and $A \subset G$ with $d^*(A) = \delta > 0$. Let $\phi_1, \phi_2, \phi_3$ be commuting endomorphisms of $G$ with finite index image such that $\phi_1 + \phi_2 + \phi_3 = 0$. Then there is an ergodic $G$-system $\mb{X} := (X, \mu, T)$ and a function $f: X \to [0, 1]$ with $\int_X f \, d\mu = d^*(A)$ such that the function $I: G \to [0,1]$ defined by
\begin{equation*}
   I(w):= UC-\lim_{g \in G} \int_X f \cdot T_{\phi_3(g)} f \cdot T_{w-\phi_2(g)} f\, d\mu
\end{equation*}
satisfies $\phi_3(\supp I) \subset \phi_1(A) + \phi_2(A) + \phi_3(A)$.
\end{proposition}
\begin{proof}
By Furstenberg's correspondence principle (for example, see \cite[Theorem 2.8]{bm}), there exists an ergodic $G$-system $(X, \mu, T)$ and a measurable set $E \subset X$ with $\mu(E) = d^*(A)$ such that for all $w_1, w_2 \in G$,
\begin{equation*}
    \mu(E \cap T_{w_1}^{-1} E \cap T_{w_2}^{-1} E) \leq d^*(A \cap (A - w_1) \cap (A - w_2)).
\end{equation*}
Letting $f = 1_E$, $w_1 = \phi_3(g)$ and $w_2 = w - \phi_2(g)$, we deduce that for all $w$ and $g \in G$,
\begin{equation*}
    \int_X f \cdot T_{\phi_3(g)} f \cdot T_{w - \phi_2(g)} f \, d \mu \leq d^*(A \cap (A - \phi_3(g)) \cap (A - (w - \phi_2(g))).
\end{equation*}
It follows that if $w \in \supp(I)$, then there are $h \in A$ and $g \in G$ such that $h$, $h + \phi_3(g)$, and $h + w - \phi_2(g)$ all belong to $A$. Therefore,
\begin{equation}
\label{eq:phi_123_h}
    \phi_3(w) = \phi_1(h) + \phi_2(h + \phi_3(g)) + \phi_3(h + w - \phi_2(g)) \in \phi_1(A) + \phi_2(A) + \phi_3(A)
\end{equation}
and this finishes our proof.
Note that in \eqref{eq:phi_123_h}, we use the fact that $\phi_2 \circ \phi_3 = \phi_3 \circ \phi_2$. 
\end{proof}

We are ready to prove \cref{th:main-density}.

\begin{proof}[Proof of \cref{th:main-density}]
By \cref{lem:correspondence_funny}, there exists an ergodic $G$-system $(X, \mu, T)$ and $f: X \to [0,1]$ with $\int_X f = d^*(A)$ such that
\begin{equation*}
    I(w) = UC-\lim_{g \in G} \int_X f \cdot T_{\phi_3(g)} f \cdot T_{w-\phi_2(g)} f\, d\mu
\end{equation*}
has $\phi_3(\supp(I)) \subset \phi_1(A) + \phi_2(A) + \phi_3(A)$. 

In view of \cref{lem:support_bohr_countable}, $\supp(I)$ contains a Bohr-$(k, \eta)$ set where $k, \eta$ only depends on $\delta$ and the indices of $\phi_i(G)$ in $G$. \cref{lem:phiB_Bohr} then implies that $\phi_3(\supp(I))$ contains a Bohr-$(k', \eta')$ set where $k', \eta'$ depends only on $\delta$ and the indices mentioned above.
\end{proof}

\section{Second correspondence principle}
\label{sec:second_correspondence} In this section we establish the second correspondence principle \cref{prop:correspondence_principle_bohr}, which is used in the proof of \cref{th:main-partition}.  This can be thought of as a special case of Propositions 3.1 and 3.2 of \cite{bg}. Here we write $bG$ for the Bohr compactification of $G$.


\begin{proposition}[Second correspondence principle]
\label{prop:correspondence_principle_bohr}
Let $K = bG$ and let $\tau: G \to K$ be the natural embedding. Let $A, B \subset G$ and let $\nu, \lambda$ be two invariant means on $G$ where $\lambda$ is extremal. Then $A + B - B$ contains $\tau^{-1}(\supp (\rho_{A}^{\nu} * \rho_{B}^{\lambda} * \rho_{-B}^{\lambda}))$.
\end{proposition}

\begin{proof}

By \cref{lem:supp_rho_A}, the Radon-Nikodym density $\rho_{A}^{\nu}$ is supported on $\overline{\tau(A)}$. Therefore the convolution $\rho_{A}^{\nu} * \rho_{B}^{\lambda}$, which is defined as 
\[
    \rho_{A}^{\nu} * \rho_{B}^{\lambda} (z) := \int_K \rho_{A}^{\nu}(x) \rho_{B}^{\lambda}(z - x) \, d m_K(x),
\]
is supported on $\overline{\tau(A)} + \overline{\tau(B)} = \overline{\tau(A + B)}$. Similarly $\rho_{A}^{\nu} * \rho_{B}^{\lambda} * \rho_{-B}^{\lambda}$ is supported on $\overline{\tau(A + B - B)}$. This, however, is weaker than the conclusion of \cref{prop:correspondence_principle_bohr} and is insufficient for our purpose.


Define $\phi, \theta: G \to [0, 1]$ by 
\begin{equation*}
    \phi(t) := 1_B *_{\lambda} 1_{-B}(t) := \int_G 1_B(x) 1_{-B}(t - x) \, d \lambda(x)
\end{equation*}
and
\begin{equation*}
    \theta(t) := 1_A *_{\nu} \phi(t) := \int_G 1_A(y) \phi(t-y) \, d \nu(y).
\end{equation*}
We can see that $\theta$ is supported on $A + B - B$. It remains to show that $\theta = (\rho_{A}^{\nu} * \rho_{B}^{\lambda} * \rho_{-B}^{\lambda}) \circ \tau$.

\begin{claim}
\label{claim:decomp_phi}
    $\phi = \eta + \psi$ where $\psi$ is a null function and $\eta:= (\rho_{B}^{\lambda} * \rho_{-B}^{\lambda}) \circ \tau$.
\end{claim}
\begin{proof}[Proof of claim]
One can verify that $\phi$ is positive definite by writing $\sum_{g,h\in G} c_g\overline{c}_h \phi(g-h)$ as $\int_G (\sum_{g} c_g1_B(x-g))\overline{\sum_{h} c_h 1_{B}(x-h)}\, d\lambda(x)= \int_G \bigl|\sum_{g} 1_B(x-g)\bigr|^2\, d\lambda(x)$ for a finite collection of coefficients $c_g\in \mathbb C$. Therefore, by the Bochner-Herglotz Theorem, $\phi$ is the Fourier transform of a positive measure $\sigma$ on $\widehat{G}$. Decomposing $\sigma = \sigma_{d} + \sigma_{c}$ where $\sigma_{d}$ is the discrete component of $\sigma$ and $\sigma_{c}$ is the continuous part, we have 
\begin{equation}\label{eq:discpluscont}     \phi = \hat{\sigma}_{d} + \hat{\sigma}_{c}.
\end{equation}
Since $\sigma_d$ has only countably many atoms, $\hat{\sigma}_{d}$ is an almost periodic function.   
On the other hand, by Wiener's lemma (see \cite[Th\'eor\`eme 16(2)]{godement}), $\int_G |\hat{\sigma}_{c}|^2 \, d \mu = 0$ for all invariant means $\mu$ on $G$. 

Now we will prove that $\hat{\sigma}_d=\eta$.  We first show that $\hat{\sigma}_d$ and $\eta$ are almost periodic functions defined by Fourier series on $G$ with absolutely summable coefficients. To see this for $\hat{\sigma}_d$, we write $\hat{\sigma}_d = \sum_{\chi\in \widehat{G}} \sigma(\{\chi\}) \chi$, where $\sum_{\chi\in \widehat{G}} \sigma(\{\chi\})$ is a convergent sum of nonnegative values. For $\eta$, note that both $\rho_{B}^{\lambda}$ and $\rho_{-B}^{\lambda}$ are in $L^2(m_K)$.  Thus, their Fourier coefficients are square-summable, and the Fourier coefficents of $\rho_B^{\lambda}*\rho_{-B}^{\lambda}$ are absolutely summable.    
To prove that $\hat{\sigma}_d=\eta$, it  therefore suffices to prove that $\hat{\sigma}_d$ and $\eta$ have the same Fourier coefficients.  This is the same as showing that $\phi$ and $\eta$ have the same Fourier coefficients, as the Fourier coefficients of $\hat{\sigma}_c$ are all $0$ (since $\hat{\sigma}_c$ is a null function). So we verify that
\begin{equation*}
    \mu(\phi \overline{\chi}) = \mu(\eta \overline{\chi})
\end{equation*}
for every invariant mean $\mu$ on $G$ and every character $\chi \in \widehat{G}$. Fix the invariant mean $\mu$, characters $\chi \in \widehat{G}$, and $\chi' \in \widehat{K}$ such that $\chi = \chi' \circ \tau$. We then have
\begin{align*}
    \mu(\phi \overline{\chi}) &= \iint_{G^2} 1_B(t) 1_{-B}(s - t) \overline{\chi(s)} \, d \lambda(t) d \mu(s) \\
    &=\iint_{G^2} (1_B \cdot \overline{\chi})(t) \cdot (1_{-B} \cdot \overline{\chi})(s - t) \, d \lambda(t) d \mu(s) \\
    &= \lambda(1_B \cdot \overline{\chi}) \lambda(1_{-B} \cdot \overline{\chi}) \;\;\; \text{(by \cref{lem:inv_mean_need_one_extremal})}\\
    &= \int_K \rho_{B}^{\lambda} \overline{\chi'} \, d m_K \cdot \int_K \rho_{-B}^{\lambda} \overline{\chi'} \, d m_K  \;\;\; \text{(by definitions of $\rho_{B}^{\lambda}$ and $\rho_{-B}^{\lambda}$)} \\
    &= \widehat{\rho_{B}^{\lambda}}(\chi') \cdot \widehat{\rho_{-B}^{\lambda}}(\chi')\\
    &= \widehat{\rho_{B}^{\lambda} * \rho_{-B}^{\lambda}}(\chi') \\
    &= \int_K (\rho_{B}^{\lambda} * \rho_{-B}^{\lambda}) \cdot \overline{\chi'} \, d m_K \\
    &= \mu(\eta \overline{\chi}) \;\; \; \text{(by the definition of $\eta$ and \cref{lem:uniqueness-Bohr})}. \qedhere
\end{align*}
\end{proof}

We are ready to prove $\theta = (\rho_{A}^{\nu} * \rho_{B}^{\lambda} * \rho_{-B}^{\lambda}) \circ \tau$. Indeed, by \cref{claim:decomp_phi},
\begin{equation*}
    \theta := 1_A *_{\nu} \phi = 1_A *_{\nu} \eta + 1_A *_{\nu} \psi 
\end{equation*}
where $\psi$ is a null function and $\eta=(\rho_{B}^{\lambda} * \rho_{-B}^{\lambda}  )\circ\tau$. For all $t \in G$, we have
\begin{equation*}
    |1_A *_{\nu} \psi(t)| \leq \nu(|-\psi_t|) = \nu(|\psi|) = 0.
\end{equation*}
Moreover, since $\eta$ is a Fourier series with absolutely summable coefficients, $1_A *_{\nu} \eta$ is as well. 
It follows that $\theta$ is almost periodic. Therefore, to show $\theta = (\rho_{A}^{\nu} * \rho_{B}^{\lambda} * \rho_{-B}^{\lambda}) \circ \tau$, it suffices to check that $\theta$ and $(\rho_{A}^{\nu} * \rho_{B}^{\lambda} * \rho_{-B}^{\lambda}) \circ \tau$ have the same Fourier coefficients. We omit the computations as they are nearly identical to the proof of \cref{claim:decomp_phi}.
\end{proof}

\section{Bohr sets in \texorpdfstring{$\phi_1(A_i) + \phi_2(A_i) - \phi_2(A_i)$}{phi1(Ai) + phi2(Ai) - phi2(Ai)}}

\label{sec:bohr_sets_partition}


In this section we prove \cref{th:main-partition}, which says that $\phi_1(A_i) + \phi_2(A_i) - \phi_2(A_i)$ contains a Bohr set for some $A_i$ in any partition $G = \bigcup_{i=1}^r A_i$. Since the proof is technical and  uses cumbersome notation, we first sketch the main idea.
Fix an invariant mean $\nu$ on $G$. The pushforwards $\phi_{1, *} \nu$ and $\phi_{2, *} \nu$ are invariant means on $H_1 = \phi_1(G)$ and $H_2 = \phi_2(G)$, respectively. Since $H_1, H_2$ are only subgroups of $G$, in order to apply the correspondence principle (\cref{prop:correspondence_principle_bohr}), we need to extend $\phi_{1, *} \nu$ and $\phi_{2, *} \nu$ to  means $\nu_1$ and $\nu_2$ on $G$. Furthermore, $\nu$ can be chosen in such a way that $\nu_2$ is extremal. Having found such extensions, \cref{prop:correspondence_principle_bohr} implies that $\phi_1(A_i) + \phi_2(A_i) - \phi_2(A_i)$ contains the preimage of the support of
\begin{equation*}
   \rho_{\phi_1(A_i)}^{\nu_1} * \rho_{\phi_2(A_i)}^{\nu_2} * \rho_{-\phi_2(A_i)}^{\nu_2},
\end{equation*}
which in turn contains a Bohr set for some $i \in [r]$ thanks to \cref{cor:right_containment} and the corresponding partition result in compact groups (Theorem \ref{th:ll_compact} (ii)) from \cite{ll}. 

The precise result we need from \cite{ll} is the following.
\begin{proposition} [{\cite[Proposition 3.4]{ll}}] \label{prop:partition}
Let $K$ be a compact abelian group and $\tilde{\phi_1}, \tilde{\phi_2}$ be commuting continuous endomorphisms on $K$ with finite index images. Suppose $\rho_1, \ldots, \rho_r: K \to [0,1]$ are measurable functions such that $\sum_{i=1}^r \rho_i \geq 1$ almost everywhere. For $w \in G$, define
\[
  R_i(w) =  \iint_{K^2}  \rho_i(\tilde{\phi_2}(v)) \rho_i(w+u) \rho_i(u + \tilde{\phi_1}(v)) \ d\mu_K(u) d\mu_K(v). 
\]
Then there are $k, \eta >0$ depending only on $[K:\tilde{\phi_1}(K)], [K:\tilde{\phi_2}(K)]$ and $r$ such that for some $i \in [r]$, the support of $R_i$ contains a Bohr-$(k,\eta)$ set. 
\end{proposition}

 We turn to the details.  The following lemma helps us extend an invariant mean on $H = \phi(G)$ to a mean on $G$ by thinking of $\ell^{\infty}(H)$ as embedded into $\ell^{\infty}(G)$ through the pullback map $\phi^*$. 
\begin{lemma}
\label{lem:existence_mean_pullback}
Let $G$ and $H$ be discrete abelian groups and $\phi:G \to H$ be a surjective homomorphism. Then for every invariant mean $\mu$ on $H$, there exists an invariant mean $\nu$ on $G$ such that $\phi_* \nu = \mu$.
\end{lemma}

\begin{proof}


First we observe that if $\nu$ is a linear functional on $\ell^\infty_\R(G)$ and $\nu(1_G) = 1$, then $\nu$ is positive if and only if $\nu(f) \geq p(f):= \inf_{x \in G} f(x)$ for all $f \in \ell^{\infty}_\R(G)$. Clearly $p$ is a concave function.

Let $V$ be the vector subspace of $\ell_\R^\infty(G)$ consisting of functions of the form $h \circ \phi$ for some $h \in \ell_\R^\infty(H)$. If $f \in V$, then by surjectivity of $\phi$, there is a unique $h \in \ell_\R^\infty(H)$ such that $f = h \circ \phi$. We have
\begin{eqnarray*}
\mu(h) &\geq& \inf_{y \in H} h(y) \qquad \textup{(since $\mu$ is an invariant mean on $H$)} \\
 & = & \inf_{x \in G} h(\phi(x)) = p(f) \qquad \textup{(since $\phi$ is surjective)}.
\end{eqnarray*}

By the Hahn-Banach theorem, the linear functional $f \mapsto \mu(h)$ on $V$ can be extended to a linear functional $\lambda$ on $\ell_\R^\infty(G)$ such that $\lambda(f) \geq p(f)$ for any $f \in \ell_\R^\infty(G)$. 
In particular, $\lambda$ is positive and $\lambda(1_G) = \lambda(1_H \circ \phi) = \mu(1_H) = 1$. We now show that $\lambda$ can be further refined to become $G$-invariant. 



We let $\eta$ be an invariant mean on $G$, and define
\[
\nu(f): = \int_{G} \lambda(f_x) \, d\eta(x) 
\]
for all $f \in \ell_\R^\infty(G)$. Then $\nu(f_g) = \nu(f)$ for all $g \in G$, since $\eta$ is translation invariant. The positivity of $\nu$ follows from the positivity of $\lambda$ and $\eta$. If $f = h \circ \phi \in V$, then $\lambda(f_g) = \mu(h_{\phi(g)}) = \mu(h)$ for all $g \in G$, so $\nu(f) = \mu(h)$. The lemma now follows, since an invariant mean is completely determined by its values on real-valued functions.
\end{proof}

If $H$ happens to be a subgroup of $G$, then another way to extend a mean on $H$ to a mean on $G$ is to consider $\ell^{\infty}(H)$ as a subset of $\ell^{\infty}(G)$ consisting of functions supported on $H$. This is the content of the next lemma.

\begin{lemma}
\label{lem:H_G_mean}
Let $H$ be a subgroup of $G$ of index $k \in \N$ and let $\mu$ be an invariant mean on $H$. There exists a unique invariant mean $\nu$ on $G$ such that 
\begin{equation*}
\nu(f) = \frac{\mu(f)}{k}
\end{equation*}
for every $f \in \ell^{\infty}(G)$ supported on $H$. Furthermore, if $\mu$ is extremal then $\nu$ is also extremal.
\end{lemma}

\begin{proof}
Let $H - g_i$ for $0 \leq i \leq k - 1$ be the cosets of $H$ in $G$ with $g_0 = 0$. We first show that an invariant mean $\nu$ satisfying the conclusion of the lemma must be unique. For a function $f$ supported on $H - g_i$, the function $f_{g_i}$ given by $x \mapsto f(x - g_i)$ is supported on $H$. Therefore, in this case, since $\nu$ is $G$-invariant, we must have
\begin{equation}\label{eq:first_define_nu}
     \nu(f) = \nu(f_{g_i}) = \frac{\mu(f_{g_i})}{k}.
\end{equation}

For an arbitrary $f \in \ell^{\infty}(G)$, define $f^i = f \cdot 1_{H - g_i}$.
Since $f = \sum_{i=0}^{k-1} f^i$, from the previous paragraph, we must have 
\begin{equation}
\label{eq:define_nu}
    \nu(f) = \sum_{i=0}^{k-1} \nu(f^i) = \frac{1}{k} \sum_{i=0}^{k-1} \mu((f^i)_{g_i}).
\end{equation}
This equation uniquely defines $\nu$.

It is easy to see that $\nu$ as defined in \eqref{eq:define_nu} is a linear functional on $\ell^{\infty}(G)$ with $\nu(1_G) = 1$. To show $\nu$ is $G$-invariant, we consider arbitrary $g \in G$ and $f \in \ell^{\infty}(G)$.
By the linearity of $\nu$ and \eqref{eq:first_define_nu},
\begin{equation}\label{eq:define_shift_nu}
    \nu(f_g) = \sum_{i=0}^{k-1} \nu((f^i)_g) = \frac{1}{k} \sum_{i=0}^{k-1} \mu(((f^i)_g)_{g_{j(i)}}) = \frac{1}{k} \sum_{i=0}^{k-1} \mu(((f^i)_{g+g_{j(i)}}). 
\end{equation}
where $j(i) \in \{0, \ldots, k-1\}$ is such that $-g_i + g + g_{j(i)}  \in H$.
For $i \in \{0, \ldots, k - 1\}$, let $h = -g_i + g + g_{j(i)}$. Since $\mu$ is $H$-invariant,
\begin{equation}\label{eq:mu_f_g_j}
    \mu(((f^i)_{g+g_{j(i)}}) = \mu(((f^i)_{g_i + h}) = \mu((f^i)_{g_i}).
\end{equation}
Relations \eqref{eq:define_nu}, \eqref{eq:define_shift_nu}, and \eqref{eq:mu_f_g_j} give $\nu(f_g) = \nu(f)$, and so $\nu$ is $G$-invariant.

Suppose $\mu$ is extremal. To show that $\nu$ is extremal, suppose $\nu = \alpha \nu_1 + (1 - \alpha) \nu_2$ where $\nu_1$ and $\nu_2$ are means on $G$ and $0 < \alpha < 1$. Restricting to $S := \{f \in \ell^{\infty}(G): f \text{ is supported on } H\}$, we get
\begin{equation*}
    \mu/k = \nu|_{S} = \alpha  \nu_1|_S + (1 - \alpha) \nu_2|_S.
\end{equation*}
Since $\mu$ is extremal, it must be that $\nu_1|_S = \nu_2|_S = \mu/k$. Due to the uniqueness of the extension of $\mu$ from $H$ to $G$, we deduce that $\nu_1 = \nu_2 = \nu$. Therefore, $\nu$ is extremal. 
\end{proof}

The next lemma shows that if $H$ is a subgroup of $G$ with finite index, then the Radon-Nikodym density associated with the mean $\mu$ on $H$ and the one associated with its extension on $G$ are the same.

\begin{lemma}
\label{lem:H_G_nu_mu}
Let $H$ be a subgroup of $G$ of index $k \in \N$. Let $K$ be a compact abelian group and $\tau: G \to K$ be a homomorphism with dense image and $K_H = \overline{\tau(H)}$. Let $B \subset H$ and $\mu$ be an invariant mean on $H$. Let $\nu$ be the extension of $\mu$ to $G$ as stated in \cref{lem:H_G_mean}. Suppose $\rho_{B}^{\nu}: K \to [0, 1]$ and $\rho_{B}^{\mu}: K_H \to [0, 1]$ are the associated Radon-Nikodym densities.
By identifying $\rho_{B}^{\mu}$ with its extension to $0$ outside of $K_H$, we have
\begin{equation*}
    \rho_{B}^{\nu} = \rho_{B}^{\mu}
\end{equation*}
$m_K$-almost everywhere.
\end{lemma}

\begin{proof}
As in the proof of \cref{lem:H_G_mean}, let $H - g_i$ for $0 \leq i \leq k -1 $ be the cosets of $H$ in $G$ with $g_0 = 0$.
Since $B \subset H$, according to \cref{lem:supp_rho_A}, both $\rho_{B}^{\nu}$ and $\rho_{B}^{\mu}$ are supported on $K_H$.
From \eqref{eq:define_nu}, for $h \in C(K)$,
\begin{equation}
\label{eq:hcirctau}
    \nu(h \circ \tau \cdot 1_B) = \frac{1}{k} \sum_{i=0}^{k-1} \mu((h \circ \tau \cdot 1_B \cdot 1_{H-g_i})_{g_i}).
\end{equation}
Since $1_B$ is supported on $H$, 
\[
    h \circ \tau \cdot 1_B \cdot 1_{H-g_i} = 0 \text{ if $i \neq 0$}.
\]
Therefore, the right hand side of \eqref{eq:hcirctau} is equal to
\begin{equation*}
    \frac{1}{k} \mu(h \circ \tau \cdot 1_B) 
\end{equation*}
which is equal to 
\begin{equation*}
    \frac{1}{k} \int_{K_H} h \cdot \rho_{B}^{\mu} \, d m_{K_H}.
\end{equation*}
It follows that
\begin{equation*}
    \int_{K} h \cdot \rho_{B}^{\nu} \, d m_{K} = \nu(h \circ \tau \cdot 1_B) = \frac{1}{k} \int_{K_H} h \cdot \rho_{B}^{\mu} \, d m_{K_H}. 
\end{equation*}
Since when restricting to $K_H$, the measure $m_K$ is equal to $\frac{1}{k} m_{K_H}$, we deduce that $\rho_{B}^{\nu} = \rho_{B}^{\mu}$.
\end{proof}




We are ready to prove \cref{th:main-partition}.  Our proof will use \cref{cor:right_containment}, applied in the case where $K_1=bG$, $K_2=\tilde{\phi}(bG)$ (where $\tilde{\phi}$ is given by \cref{lem:Embeddings}(ii)), and $\tau_1=\tau_2=\tau=$ the canonical embedding of $G$ into $bG$.  In order to verify that the hypotheses of \cref{cor:right_containment} are satisfied, we want to know that every character $\psi$ of $\phi(G)$ can be written in the form $\chi' \circ \tau$, where $\chi'$ is a character of $\tilde{\phi}(bG)$. 
This is the case, as every $\psi\in \widehat{\phi(G)}$ can be extended to a character $\psi_0\in \widehat{G}$, and $\psi_0=  \chi_0\circ \tau$ for some $\chi_0 \in \widehat{bG}$.  Let $\chi':=\chi_0|_{\tilde{\phi}(bG)}$.  We claim that $\chi'\circ \tau = \psi$.  To see this, note that $\chi_0\circ \tau = \psi_0$, so $(\chi_0\circ \tau)|_{\phi(G)}=\psi_0|_{\phi(G)}=\psi$.  Finally, note that $\tau(\phi(G))\subset \tilde{\phi}(bG)$, since $\tilde{\phi}\circ \tau = \tau\circ \phi$.  Thus $(\chi_0\circ \tau)|_{\phi(G)}=\chi'\circ \tau$.


\begin{proof}[Proof of \cref{th:main-partition}]
Let $H_1 = \phi_1(G)$ and $H_2 = \phi_2(G)$.
Let $\mu$ be an extremal invariant mean on $H_2$. By \cref{lem:existence_mean_pullback}, there exists an invariant mean $\nu$ on $G$ such that the pushforward $\phi_{2, *} \nu$ is equal to $\mu$. In view of \cref{lem:H_G_mean}, $\phi_{1,*} \nu$ can be extended canonically from $H_1$ to a mean $\nu_1$ on $G$ such that
\begin{equation*}
    \nu_1(f) = \frac{(\phi_{1, *} \nu)(f)}{[G:H_1]}
\end{equation*}
for every $f \in \ell^{\infty}(G)$ supported on $H_1$.
Likewise, extend $\mu = \phi_{2,*} \nu$ from  $H_2$ to a mean $\nu_2$ on $G$. Since $\mu$ is extremal, $\nu_2$ is extremal; however, $\nu_1$ may not be extremal.

Let $A \subset G$, $K = bG$ and $\tau: G \to K$ be the natural embedding. By \cref{prop:correspondence_principle_bohr} and because $\nu_2$ is extremal, the sumset $\phi_1(A) + \phi_2(A) - \phi_2(A)$ contains
\begin{equation*}
    \tau^{-1}(\supp \rho_{\phi_1(A)}^{\nu_1} * \rho_{\phi_2(A)}^{\nu_2} * \rho_{\phi_2(-A)}^{\nu_2}).
\end{equation*}
In light of \cref{lem:H_G_nu_mu},  \begin{equation*}
   \rho_{\phi_j(A)}^{\nu_j} = \rho_{\phi_j(A)}^{\phi_{j,*} \nu} 
\end{equation*}
where we identify $\rho_{\phi_j(A)}^{\phi_{j,*} \nu}$ with its extension to $0$ outside of $\phi_j(K)$. It follows that $\phi_1(A) + \phi_2(A) - \phi_2(A)$ contains
\begin{equation*}
    \tau^{-1}(\supp \rho_{\phi_1(A)}^{\phi_{1, *} \nu} * \rho_{\phi_2(A)}^{\phi_{2, *} \nu} * \rho_{\phi_2(-A)}^{\phi_{2, *} \nu}).
\end{equation*}

For $j \in \{1, 2\}$, let $\tilde{\phi}_j: K \to K$ be continuous homomorphism such that $\tilde{\phi}_j \circ \tau = \tau \circ \phi_j$. Then $\tilde{\phi}_1 \circ \tilde{\phi}_2 \circ \tau =
 \tau \circ \phi_1 \circ \phi_2 =  \tau \circ \phi_2 \circ \phi_1 = \tilde{\phi}_2 \circ \tilde{\phi}_1 \circ \tau$. It follows that $\tilde{\phi}_1$ and $ \tilde{\phi}_2$ commute since $\tau(G)$ is dense in $K$. By Lemma \ref{lem:Embeddings}, $[K: \tilde{\phi}_j(K)] \leq [G: \phi_j(G)]$ is finite. 

For ease of notation, we write \[
f := \rho_{\phi_1(A)}^{\phi_{1, *} \nu}, \quad g := \rho_{\phi_2(A)}^{\phi_{2, *} \nu} \quad \textup{and } h:=\rho_{\phi_2(-A)}^{\phi_{2,*} \nu}.
\]
Note that $f, g, h$ are nonnegative.
\begin{claim}
The support of $f * g * h$ contains the support of $S: K \to [0, 1]$ defined by
\[
    S(w) := \iint_{K^2} f(\tilde{\phi_1} \circ \tilde{\phi_2}(v)) \cdot g(w + \tilde{\phi_2}(u)) \cdot h(- \tilde{\phi_2}(u) - \tilde{\phi_2} \circ \tilde{\phi_1}(v)) \ d m_K(u) d m_K(v).
\]
\end{claim}
\begin{proof}[Proof of Claim]
Note that by \cite[Lemma 2.6]{ll}, $\tilde{\phi_1} \circ \tilde{\phi_2} (K)$ has finite index in $K$. We recall \cite[Lemma 2.8]{ll}, which says that if $f$ is a nonnegative function on a compact abelian group $K$, $\phi$ is a continuous endomorphism on $K$ and $m = [K: \phi(K)] < \infty$, then
\[
\int_K f(\phi(x)) \, d\mu_K (x) \leq m \int_K f(x) \, d\mu_K (x).
\]
 By two applications of this fact, we have 
 \begin{eqnarray*}
 S(w) & \leq & [K:\tilde{\phi_2}(K)] \iint_{K^2} f(\tilde{\phi_1} \circ \tilde{\phi_2}(v)) \cdot g(w + u) \cdot h(- u - \tilde{\phi_2} \circ \tilde{\phi_1}(v)) \ d m_K(u) d m_K(v) \\
 &\leq& [K:\tilde{\phi_2}(K)] \cdot [K: \tilde{\phi_1} \circ  \tilde{\phi_2}(K)] \iint_{K^2} f(v) \cdot g(w + u) \cdot h(- u - v) \ d m_K(u) d m_K(v) \\
& = & [K:\tilde{\phi_2}(K)] \cdot [K: \tilde{\phi_1} \circ  \tilde{\phi_2}(K)] \cdot f*g*h(w),
\end{eqnarray*}
thus proving the claim.
\end{proof} 

By \cref{cor:right_containment} we have
\begin{align}\label{eq:p1p2}
    f (\tilde{\phi}_1 \circ \tilde{\phi}_2(v)) &\geq \rho_{A}^{\nu}(\tilde{\phi}_2(v)), \\
    \label{eq:p2p2}
    g (\tilde{\phi}_2(w) + \tilde{\phi}_2(u)) &\geq \rho_{A}^{\nu}(w + u),   
\end{align}
and
\begin{equation}
\label{eq:p2p1}
    h ( - \tilde{\phi}_2(u) - \tilde{\phi}_2 \circ \tilde{\phi}_1(v))) \geq \rho_{A}^{\nu}(u + \tilde{\phi}_1(v)).
\end{equation}

Therefore
\begin{equation}\label{eq:R1A}
    S(\tilde{\phi}_2(w)) \geq R_{A}(w)
\end{equation}
for all $w \in K$, where
\begin{equation*}
    R_{A}(w) := \iint_{K^2} \rho_{A}^{\nu}(\tilde{\phi}_2(v)) \rho_{A}^{\nu}(w + u) \rho_{A}^{\nu}(u + \tilde{\phi}_1(v)) \, d m_K(u) dm_K(v).
\end{equation*}
Combining (\ref{eq:p1p2}) - (\ref{eq:R1A}), we get that for all $A \subset G$, the sumset $\phi_1(A) + \phi_2(A) - \phi_2(A)$ contains $\tau^{-1}(\tilde{\phi}_2(\supp R_{A}))$.

As a consequence, we have for each partition $G = \bigcup_{i=1}^r A_i$ and each $i \in [r]$,
\begin{equation*}
    \phi_1(A_i) + \phi_2(A_i) - \phi_2(A_i) \supset \tau^{-1}(\tilde{\phi}_2(\supp R_{A_i})).
\end{equation*}
By \cref{cor:sum_rho_1}, $\sum_{i=1}^r \rho_{A_i}^{\nu} = 1$ almost everywhere. Therefore, in view of Proposition \ref{prop:partition}, for some $i \in [r]$, the support of $R_{A_i}$ contains a Bohr-$(k, \eta)$ set $B \subset K$ where $k, \eta$ depend only on $r$ and the indices $[K:\tilde{\phi}_1(K)], [K:\tilde{\phi}_2(K)]$. 

By Lemma \ref{lem:phiB_Bohr}, $\tilde{\phi}_2(B)$ is a Bohr-$(k', \eta')$ set where $k', \eta'$ depend only on $k, \eta$ and $[K:\tilde{\phi}_2(K)]$. \cref{lem:from_kronecker_up} then implies that $\tau^{-1}(\tilde{\phi}_2(B))$ contains a Bohr-$(k', \eta')$ set and our proof finishes.
\end{proof}

\section{Third correspondence principle}

\label{sec:B+C+A_i_correspondence}

In this section we derive a correspondence principle for $B + C + A_i$.  Assuming only that the summands $A, B, C$ have positive upper Banach density, we cannot guarantee that $A+B+C$ is a Bohr set, a translate of a Bohr set, or even that $A+B+C$ is syndetic.\footnote{In every countably infinite abelian group, there are sets $D,E$ with positive upper Banach density where $D+E$ is not syndetic, and Proposition 6.2 of \cite{BBF} produces sets $A,B,C$ having positive upper Banach density where $A+B+C\subset D+E$.}  Under the stronger assumption that $A$ and $B$  have positive upper Banach density and that $C$ is syndetic, \cite{bfw-bohr} proves (for the ambient group $\mathbb Z$) that $A+B+C$ contains a translate of a Bohr set.  Our Theorem \ref{thm:B+C+A_i_discrete} has a similar, but weaker hypothesis: partitioning $G$ as $A_1\cup \cdots \cup A_r$, it is possible that none of the $A_i$ are syndetic.  Of course, one of the $A_i$ must be piecewise syndetic (\cite{brown-PWsynd}, \cite{hindman-strauss-book}).

 Proposition \ref{prop:B+C-A} says that when $A, B, C\subset G$ with $d^*(B), d^*(C)>0$, the sumset $B+C+A$ can be modeled by a convolution $h_B*h_C*h_A$ on a compact group $K$, where $\int h_B\, dm_K \geq d^*(B)$ and $\int h_C\, dm_K\geq d^*(C)$.  In this correspondence principle, the hypothesis $d^*(A)>0$ is not strong enough to guarantee that $h_A$ is nonzero.  However, assuming that $G=A_1\cup \cdots \cup A_r$, we will be able to conclude that $\sum_{i=1}^r h_{A_i} \geq 1$ almost everywhere and this suffices to give an useful bound on the $h_B * h_C * h_{A_i}(0)$ for some $i \in [r]$.




\begin{definition}
Let $A,B\subset G$.  We write $A\prec B$ if for all finite subsets $A'\subset A$, there exists $t\in G$ such that $A'+t\subset B$.  In this case, we say that $A$ is \emph{finitely embeddable in $B$}. 
\end{definition}



%


The following lemma is implicit in \cite{griesmer-small-sum-pairs} and to some extent in \cite{griesmer-dense-set}.  A similar statement for amenable groups can be obtained from Propositions 1.10 and 1.11 in \cite{bf-advances}.

\begin{lemma}\label{lem:PWBohrSumsets}
Let $B, C \subset G$. There exist a compact abelian group $K$, a homomorphism $\tau: G \to K$ for which $\tau(G)$ is dense in $K$, functions $h_B, h_C: K \to [0, 1]$ such that 
\begin{enumerate}[label=(\roman*)]
    \item $\int_K h_B \, d m_K = d^*(B)$ and $\int_K h_C \, d m_K = d^*(C)$, and
    
    \item $\{g \in G: h_B * h_C(\tau(g)) > 0\} \prec B+C$.

\end{enumerate}
\end{lemma}

\begin{remark}
Readers familiar with Furstenberg's correspondence principle and Kronecker factors may appreciate the following additional detail: to obtain the group $K$, one may apply the Furstenberg correspondence principle to find ergodic measure preserving systems $\mb X_B=(X_B,\mu_B,T_B)$ and $\mb X_C=(X_C,\mu_C,T_C)$ modeling $B$ and $C$, with corresponding Kronecker factors $\mb K_B=(K_B,m_{K_B},R_B)$ and $\mb K_C=(K_C,m_{K_C},R_C)$.  The groups $K_B$ and $K_C$ are the respective duals of the eigenvalue groups $\mathcal E(\mb X_B)$ and $\mathcal E(\mb X_C)$ of $\mb X_B$ and $\mb X_C$ (as described by \cref{lem:rotations}).  The group $K$ may be realized as the phase space of the maximal common factor of $\mb K_B$ and $\mb K_C$, or, equivalently, as the dual of $\mathcal E(\mb X_B)\cap \mathcal E(\mb X_C)$.
\end{remark}

\begin{proof}
By \cite[Lemma 2.8]{griesmer-small-sum-pairs}, there is an ergodic measure preserving $G$-system $(X,\mu,T)$, where $X$ is a compact metric space, and a clopen set $O_C\subset X$ with $\mu(O_C)=d^*(C)$ such that for all $x\in X$, 
\begin{equation}
\label{eqn:BOAprecAB1}
    \{g\in G: T_gx\in \bigcup_{b \in B} T^b O_C\} \prec B + C.
\end{equation}
By \cite[Lemma 4.1]{griesmer-small-sum-pairs}, there is a group rotation factor $(K, m_K,R)$ of $(X,\mu,T)$ with factor map $\pi:X\to K$ and a homomorphism $\tau: G \to K$ with dense image such that
\begin{equation}
\label{eq:TbOC}
\bigcup_{b \in B} T^b O_C \supset \pi^{-1}(J) \; \text{ up to a set of } \mu\text{-measure } 0,
\end{equation}
where $J := \supp (f_B * f_C)$ for some functions $f_B, f_C:K\to [0,1]$ with $\int_K f_B\, dm_K = d^*(B)$ and $\int_K f_C \, dm_K=d^*(C)$. 

Note that for $\mu$-almost every $x \in X$, $R_g \pi(x) = \pi(T_g x)$. Therefore, if $R_g (\pi(x)) \in J$, then $T_g x \in \pi^{-1}(J)$.  Thus, from \eqref{eq:TbOC}, for $\mu$-almost every $x\in X$, we have 
\[
    \text{if } R_g(\pi(x))\in J \text{ then } T_g x\in \bigcup_{b \in B} T^b O_C.
\]
Fix such an $x$. Then 
\[
    \{g \in G: f_B*f_C(\pi(x)+\tau(g))>0\} \subset \{g \in G: T_g x\in \bigcup_{b \in B} T^b O_C\}.
\]
The relation \eqref{eqn:BOAprecAB1} then implies $\{g \in G: f_B*f_C(\pi(x)+\tau(g))>0\} \prec B+C$. By defining functions $h_B, h_C$ as $h_B(t) := f_B(t + \pi(x))$ and $h_C = f_C$, we obtain our conclusion.

\end{proof}






\begin{lemma}\label{lem:PrecConvolution}
Let $K$ be a compact metrizable abelian group and $\tau:G\to K$ be a homomorphism with dense image. Let $h: K\to [0,1]$ be continuous and let $A_h:=\{g \in G:h(\tau(g))>0\}$. If $A_h\prec D$, then there is a translate $h'$ of $h$ and an invariant mean $\lambda$ on $G$ such that 
\begin{equation*}\label{eqn:1C>h'} 1_{D}*_{\lambda } q \geq h'\circ \tau *_{\lambda } q 
\end{equation*}
for all $q: G\to [0,1]$.
\end{lemma}

\begin{proof} Let $(F_N)_{N\in \mathbb N}$ be a F{\o}lner sequence for $G$.  Since $F_N\cap A_h\subset A_h$ and $A_h\prec D$, we may choose, for each $N\in \mathbb N$, a $t_N\in G$ so that $(F_N\cap A_h)+t_N\subset D$.  Note that $(F_N+t_N)_{N \in \N}$ is also a F{\o}lner sequence. Passing to a subsequence if necessary, we assume $\tau(t_N)$ converges to a point $k_0$ in $K$.  Let $h'(k) = h(k-k_0)$ for $k \in K$, so that $h(k-\tau(t_N))$ converges uniformly to $h'(k)$.

Define a sequence of functions $p_N: F_N+t_N\to [0,1]$ by $p_N(g+t_N)=h(\tau(g))$.   Since $h(\tau(g))=0$ for each $g\in (F_N\setminus A_h)$, and $F_N\cap A_h + t_N\subset D$, we have $1_D(g)\geq p_N(g)$ for all $g\in F_N+t_N$.

For each $N\in \mathbb N$ and each $q:G\to [0,1]$ we have
\begin{equation}\label{eqn:FNtoLambda}
\begin{split}\frac{1}{|F_N|}\sum_{g\in F_N+t_N} 1_D(g)q(t-g)
&\geq \frac{1}{|F_N|}\sum_{g\in F_N+t_N} p_N(g) q(t-g)\\
&= \frac{1}{|F_N|}\sum_{g\in F_N+t_N} h(\tau(g)-\tau(t_N))q(t-g).\end{split}
\end{equation}
For each $N$, let $\lambda_N$ be the linear functional on $\ell^\infty(G)$ defined by $\lambda_N(f):=\frac{1}{|F_N|}\sum_{g\in F_N + t_N} f(g)$. Let $\lambda$  be a linear functional on $\ell^\infty(G)$ that is a weak$^*$ limit point of the sequence $\lambda_N$ (meaning that for all $f\in \ell^\infty(G)$, all $\varepsilon>0$, and all $M\in \mathbb N$ there is an $N>M$ such that $|\lambda(f)-\lambda_N(f)|<\varepsilon$). In other words, $\lambda \in \bigcap_{M=1}^\infty \overline{\{\lambda_N: N>M\}}.$

Since $h(k-\tau(t_N))$ converges uniformly in $N$ to $h(k-k_0)=h'(k)$, (\ref{eqn:FNtoLambda}) implies $1_D*_{\lambda}q(t) \geq  h'\circ \tau *_{\lambda} q(t)$ for all $t\in G$. \end{proof}

\begin{lemma}
\label{lem:D-A}
Let $K$ be a compact abelian group and $\tau:G\to K$ be homomorphism with dense image. Let $h: K\to [0,1]$ be a continuous function and $\lambda$ be an invariant mean on $G$.  Then for every $A \subset G$,
\begin{equation*}
    (h\circ \tau)*_{\lambda} 1_{A} = (h*\rho^{\lambda}_{A})\circ \tau,
\end{equation*}
where $\rho_{A}^{\lambda}$ is defined in \cref{def:rn}.
\end{lemma}


\begin{proof}
Approximating $h$ by trigonometric polynomials, it suffices to prove the statement for the special case where $h$ is a trigonometric polynomial.  By linearity, we may assume $h=\chi \in \widehat{K}$.  For such $\chi$, we have
\begin{align*}
 (\chi\circ \tau)*_{\lambda} 1_{A}(g) &:= \int_G \chi \circ \tau(x) \cdot 1_{A}(g-x)\, d\lambda(x)\\
 &= \int_G \chi \circ \tau(g+x) 1_{A}(-x)\, d\lambda(x)\\
 &= \chi\circ \tau(g) \int_G  \chi \circ \tau(x) \cdot 1_{A}(-x)\, d\lambda(x)\\
 &= \chi\circ \tau(g) \int_G \chi\circ \tau \cdot 1_{-A}\, d\lambda\\
 &= \chi\circ \tau(g) \int_K \chi \cdot \rho_{-A}^\lambda\, dm_K.
\end{align*}

Computing $\chi*\rho_{A}^\lambda(t)$ for $t\in K$, we get
\begin{align*}
\chi*\rho_{A}^\lambda(t) &= \int_K \chi(z)\rho_{A}^\lambda(t-z)\, dm_K(z) \\
&= \int_K \chi(z+t) \rho_{A}^\lambda(-z)\, dm_K(z)\\
&= \chi(t) \int_K \chi(z) \rho_{-A}^\lambda(z)\, dm_K(z)\\
&= \chi(t) \int_K \chi \cdot \rho_{-A}^\lambda\, dm_K.
\end{align*}
Substituting $\tau(g)$ for $t$, we get
\[
    (\chi \circ \tau) *_{\lambda} 1_A(g) = (\chi * \rho_{A}^{\lambda}) (\tau(g)),
\]
completing the proof. \end{proof}



Combining Lemmas \ref{lem:PWBohrSumsets}, \ref{lem:PrecConvolution} and \ref{lem:D-A}, we have a proposition which serves as a correspondence principle for $B + C + A_i$.

\begin{proposition}[Third correspondence principle]
\label{prop:B+C-A}
Let $B, C \subset G$. There exist a compact abelian group $K$, a homomorphism $\tau: G \to K$ with dense image, measurable functions $h_B, h_C: K \to [0, 1]$ and an invariant mean $\lambda$ on $G$ such that
\begin{enumerate}[label=(\roman*)]
    \item $\int_K h_B \, d m_K = d^*(B)$ and $\int_K h_C \, d m_K = d^*(C)$, 
    \item for all $A \subset G$, 
    \[
     B + C + A \supset \tau^{-1} (\supp( h_B * h_C * \rho_{A}^{\lambda})). 
    \]
\end{enumerate}
\end{proposition}

\begin{remark}
The invariant mean $\lambda$ depends on $B$ and $C$; it may not realize the upper Banach density of $A$. In particular, it is possible that $\lambda(A) = 0$ while $d^*(A) > 0$. 
\end{remark}

\begin{proof}
In view of \cref{lem:PWBohrSumsets}, there are a compact abelian group $K$, homomorphism $\tau: G \to K$ with dense image, measurable functions $h_B, h_C:K\to [0,1]$ with $\int h_B\, dm_K = d^*(B)$, $\int h_C\, dm_K=d^*(C)$ such that
\begin{equation*}
\label{eqn:BCFN}
    \{g \in G: h_B * h_C(\tau(g)) > 0\} \prec B+C.
\end{equation*}
We now apply Lemma \ref{lem:PrecConvolution} with $h_B*h_C$ in place of $h$: there is an invariant mean $\lambda$ on $G$ such that 
\begin{equation}
\label{eq:B+C_1_Ah'}
    1_{B + C} *_{\lambda} 1_A \geq h' \circ \tau *_{\lambda} 1_A,
\end{equation}
where $h'$ is a translate of $h_B*h_C$.  

By \cref{lem:D-A},
\begin{equation}
\label{eq:h'circ}
    h' \circ \tau *_{\lambda} 1_A = (h' * \rho_{A}^{\lambda}) \circ \tau.
\end{equation}
Note that $B + C + A$ contains the support of $1_{B + C} *_{\lambda} 1_A$ and $h'$ can be written as $h_B' * h_C$ where $h_B'$ is a translate of $h_B$. Therefore, \eqref{eq:B+C_1_Ah'} and \eqref{eq:h'circ} imply
\[
    B + C + A \supset \{g \in G: h_B' * h_C * \rho_{A}^{\lambda}(\tau(g)) > 0\}
\]
and this proves our proposition.  \end{proof}



\section{Bohr sets in \texorpdfstring{$B + C + A_i$}{B + C + Ai}}

\label{sec:proof_B+C+A_i}

The next proposition establishes the existence of Bohr sets in $B + C + A_i$ in compact abelian groups.

\begin{proposition}
\label{thm:B+C+A_i_compact}
Let $\delta_1, \delta_2 >0$ and $r\in \mathbb N$.  There are constants $\eta>0$ and $k\in \mathbb N$ such that the following holds: Let $K$ be a compact abelian group with probability Haar measure $m_K$ and let $f, g: K \to [0, 1]$ be measurable functions such that $\int_K f \, d m_K \geq \delta_1$ and $\int_K g \, d m_K \geq \delta_2$. For $i \in [r]$, let $h_i: K \to [0, 1]$ be measurable functions such that $\sum_{i=1}^r h_i = 1$ $m_K$-almost everywhere. Then for some $i \in [r]$, the support of $f * g * h_i$ contains a Bohr-$(k,\eta)$ set.
\end{proposition}  

\begin{proof}
The proof is similar to an argument used in  \cite{ll} (Part I of this series). 
Since $\sum_{i=1}^r h_i=1$ almost everywhere, we have
\[
    f*g* \left(\sum_{i=1}^r h_i \right)(x) = f * g * 1_K (x) = \int_K f \, d m_K \cdot \int_K g \, d m_K \geq \delta_1 \delta_2
\]
for all $x \in K$. Therefore, by the pigeonhole principle, there exists $i \in [r]$ such that $f*g*h_i(0) \geq \delta_1 \delta_2/r.$

By \cite[Lemma 2.12]{ll}, we have
\begin{align*}
\left| f*g*h_i(t) - f*g*h_i(0) \right| & = \left| \iint_{K^2} (g(x)-g_t(x)) f(y) h_i(-x-y) \, dm_K(x) dm_K(y)  \right| \\
& \leq \| \widehat{g} - \widehat{g_t} \|_\infty \|f\|_2 \| h_i\|_2 \\
& \leq  \| \widehat{g} - \widehat{g_t} \|_\infty,
\end{align*}
where $g_t(x) = g(t+x)$.  
Hence $f*g*h_i(t) > \frac{\delta_1 \delta_2}{2r}$ whenever  $\| \widehat{g} - \widehat{g_t} \|_\infty < \frac{\delta_1 \delta_2}{2r}$. By \cite[Lemma 2.1]{ll}, the set of those $t$ contains a 
Bohr-$(k, \frac{\delta_1 \delta_2}{2r})$ set $B$ with $k \leq \frac{16r^2}{(\delta_1 \delta_2)^2}$.
\end{proof}




We are ready to prove \cref{thm:B+C+A_i_discrete}.
\begin{proof}[Proof of \cref{thm:B+C+A_i_discrete}]
By \cref{prop:B+C-A}, there exist a compact abelian group $K$, a homomorphism $\tau: G \to K$ with dense image, measurable functions $h_B, h_C: K \to [0, 1]$ and an invariant mean $\lambda$ on $G$ such that
\begin{enumerate}[label=(\roman*)]
    \item $\int_K h_B \, d m_K = d^*(B)$ and $\int_K h_C \, d m_K = d^*(C)$, 
    \item for all $i \in [r]$, $B + C + A_i \supset \tau^{-1} (\supp( h_B * h_C * \rho_{A_i}^{\lambda})).$
\end{enumerate}
In light of \cref{cor:sum_rho_1}, $\sum_{i=1}^r \rho_{A_i}^{\lambda} = 1$ almost everywhere. Therefore, by \cref{thm:B+C+A_i_compact}, there exist $k$ and $\eta$ depending only on $\delta$ and $r$ such that the support of $h_B * h_C * \rho_{A_i}^{\lambda}$ contains a Bohr-$(k, \eta)$ set in $K$ for some $i \in [r]$. \cref{lem:from_kronecker_up} then implies that $B + C + A_i$ contains a Bohr-$(k, \eta)$ set in $G$. 
\end{proof}

\begin{remark}
\label{rem:pwbohr-not-quant}
The proof of \cref{thm:B+C+A_i_discrete} follows a general phenomenon: if $D \subset G$ is a piecewise Bohr set, then for any partition $G = \bigcup_{i=1}^r A_i$, there is an $i \in [r]$ such that $D + A_i$ contains a Bohr set. However, if we did not know that $D$ has the form $B + C$, it is impossible to give any quantitative bounds on the rank and radius of the Bohr set in $D + A_i$. This necessitates the presence of triple sum $B + C + A_i$ in \cref{thm:B+C+A_i_discrete}.
\end{remark}


\section{Open questions}
\label{sec:open_question}

In the proofs of Theorems \ref{th:main-density} and \ref{th:main-partition}, the assumption that $\phi_1, \phi_2, \phi_3$ commute is used to provide a parameterized solution to the relation $w \in \phi_1(A) + \phi_2(A) + \phi_3(A)$. This concern raises the question:

\begin{question}
    Can the assumption that the $\phi_j$ commute in Theorems \ref{th:main-density} and \ref{th:main-partition} be omitted?
\end{question}

The Bohr sets in \cref{thm:B+C+A_i_compact} and \cref{thm:B+C+A_i_discrete} have the same rank $k$ and radius $\eta$. \cref{thm:B+C+A_i_compact} gives $k \ll \alpha^{-6}$ and $\eta \gg \alpha^3$, where $\alpha = (\delta_1 \delta_2 r^{-1})^{1/3}$. If we are only interested in \textit{translates} of Bohr sets (i.e., Bohr neighborhoods of some element), then better bounds are available. A result of Sanders \cite[Theorem 2.4]{SandersAdditiveStructure} implies that there exists $i$ such that $B+C+A_i$ contains a translate of a Bohr-$(k,\eta)$ set with $k \ll \alpha^{-1}$ and $\eta \geq \exp\left( - c \alpha^{-1} \log \alpha^{-1} \right)$, for some absolute constant $c$. We ask the following.

\begin{question} \label{q:optimal}
Is it possible to improve on $k$ and/or $\eta$ in \cref{thm:B+C+A_i_discrete}? Can we take $k \ll \alpha^{-1}$?
\end{question}

In the spirit of Ruzsa and Hegyv\'ari's result \cite{hr} on Bohr sets in $A+A-A-a$ mentioned in the introduction, we ask whether the Bohr set in \cref{thm:B+C+A_i_discrete} can be given by a fixed element of $C$. More precisely:
\begin{question}
If $B, C \subset G$ with $d^*(B), d^*(C)>0$ and $G = \bigcup_{i=1}^r A_i$, must there exist $c \in C$ and $i \in [r]$ such that $B+c+A_i$ contains a Bohr set?
\end{question}

The proof of \cref{thm:B+C+A_i_discrete} uses the fact that $D := B + C$ is a piecewise Bohr set to deduce the Bohr structure in $D + A_i$. It is natural to ask besides piecewise Bohr, what other conditions on $D$ guarantee the existence of a Bohr set in $D + A_i$.

\begin{question}\label{q:SyndeticSummand}
What is a sufficient condition on $D \subset G$ so that for any partition $G= \bigcup_{i=1}^r A_i$, there is $i \in [r]$ such that $D+A_i$ is Bohr set (or a translate of a Bohr set)?  In particular, does the assumption that $D$ is piecewise syndetic or $d^*(D)>0$ suffice? What if $G = \Z$ and $D = \mathbb{P}$ (the set of primes) or $D = \{n^2: n \in \N\}$?
\end{question}

Our \cref{th:main-density} generalizes \cref{th:br} in two ways: replacing the ambient group $\mathbb Z$ with an arbitrary countable abelian group, and replacing the endomorphisms $g\mapsto s_ig$ with commuting endomorphisms having finite index image.  The main result of \cite{griesmer-br} generalizes \cref{th:br} in a different way: the endomorphisms still have the form $g\mapsto s_i g$, but more summands are considered.  The following conjecture is a natural joint generalization of these results.

\begin{conjecture}\label{conj:Full}
Let $G$ be a (not necessarily countable) abelian group, let $d\geq 3$, let $\phi_1,\dots, \phi_d$ be endomorphisms of $G$ such that $[G:\phi_j(G)]<\infty$ for each $j$, and such that $\phi_1+\cdots +\phi_d=0$. Then for all $A\subset G$ with $d^*(A)>0$, the sumset $\phi_1(A)+\cdots + \phi_d(A)$ contains a Bohr set with rank and radius depending only on $d^*(A)$ and the indices $[G:\phi_j(G)]$.
\end{conjecture}

Defining endomorphism $\psi: G \to G$ by $\psi(g) := \sum_{j=3}^d \phi_j(g)$. Then $\phi_1 + \phi_2 + \psi = 0$ and 
\[
     \sum_{j=1}^d \phi_j(A) \supset \phi_1(A) + \phi_2(A) + \psi(A).
\]
Therefore, if $[G:\psi(G)]$ is finite, then \cref{conj:Full} immediately follows from \cref{th:main-density}. However, it is not true in general that $\psi(G)$ has finite index (for example, take $d = 4$, $\phi_3 = - \phi_4$), and so \cref{conj:Full} is genuinely interesting. It may be necessary to impose some additional hypotheses on the $\phi_j$; see \cite[Section 4]{griesmer-br} for more discussion.

Along the same lines, we have the following conjecture for partition that extends \cref{th:main-partition}.

\begin{conjecture}
Let $G$ be a (not necessarily countable) abelian group, let $d\geq 3$ and let $\phi_1,\dots, \phi_d$ be endomorphisms of $G$ such that $[G:\phi_j(G)]<\infty$ for each $j$. Suppose $\sum_{j \in S} \phi_j = 0$ for some non-empty subset $S \subset [d]$.
Then for every finite partition $G = \bigcup_{i=1}^r A_i$, there exists $i \in [r]$ such that $\sum_{j=1}^d \phi_j(A_i)$
contains a Bohr-$(k, \eta)$ set, where $k$ and $\eta$ depend only on $r$ and the indices $[G: \phi_j(G)]$.
\end{conjecture}









\end{document}